\documentclass[a4paper, 11 pt, twoside]{amsart}
\usepackage{srollens-en}[2011/03/03]
\usepackage[left = 3.5cm, right = 3.5cm, headsep = 6mm,
footskip = 10mm, top = 35mm, bottom = 35mm, footnotesep=5mm, headheight =
2cm]{geometry}

\usepackage{enumitem}
\newlist{prooflist}{description}{1}
\setlist[prooflist]{font=\normalfont \itshape, labelindent = \parindent, leftmargin = 0pt}

\tikzset{
curve/.style = { every node/.style = { font = \scriptsize}, thick},
E1/.style = { MidnightBlue, curve},
Ei/.style = { NavyBlue, curve},
branch/.style = { Plum, curve,every node/.style = { font = \scriptsize}},
fibre/.style = {Gray,curve},
F/.style = {Orange, curve , thin},
2F/.style = {White!70!Orange,   curve, very thick},
GG/.style = {PineGreen,curve},
G/.style = {PineGreen,curve,  thin},
2G/.style = {White!70!PineGreen,  curve, very thick},
pfeil/.style = {->, every node/.style = { font = \scriptsize}},
sing/.style = {->, every node/.style = { font = \scriptsize}},
}

\usepackage{comment}

\usetikzlibrary{patterns,decorations.pathreplacing}

\usepackage[unicode,bookmarks, pdftex]{hyperref}
\hypersetup{colorlinks=true,citecolor=NavyBlue,linkcolor=NavyBlue,urlcolor=Orange, pdfpagemode=UseNone, breaklinks=true}

\newcommand{\tsum}{\textstyle\sum}
\DeclareMathOperator{\wSym}{\mathsf {wSym}}

\newcommand{\newnotion}[1]{\textbf{#1}}

\title[Birational models of marked elliptic surfaces]{Explicit birational models of marked elliptic surfaces of relative degree two}

\author{S\"onke Rollenske}
\address{S\"onke Rollenske\\FB 12/Mathematik und Informatik\\
Philipps-Universit\"at Marburg\\
Hans-Meerwein-Str. 6\\
35032 Marburg\\
Germany}
\email{rollenske@mathematik.uni-marburg.de}
 
\author{Anna Ulivi}
\address{Anna Ulivi \\ Dipartimento di Matematica \\ Università degli Studi di Genova \\ via Dodecaneso 35
16146 Genova \\ Italy}
\email{anna.ulivi@edu.unige.it}

\author{Aline Zanardini}
\address{Aline Zanardini \\Institut de Mathématiques \\ École polytechnique fédérale de Lausanne \\ Station 8, 1015 Lausanne \\ Switzerland}
\email{aline.zanardini@epfl.ch}

\begin{document}
\begin{abstract}
We introduce explicit construction methods for two distinct birational models of elliptic surfaces equipped with a bisection --- or, more generally, a relative polarisation of degree two --- and double fibres. For Halphen surfaces of index two and  Enriques surfaces, we illustrate how these give simple descriptions inside a toric variety.
\end{abstract}

\maketitle
\setcounter{tocdepth}{1}
\tableofcontents

\section{Introduction}
One of the central goals in birational geometry is to construct and describe distinguished
birational models for algebraic varieties, which are useful for moduli problems. While every birational class of curves contains a unique smooth projective model, in higher dimensions, one is led to the notion of minimal and canonical models, in which mild singularities may appear. In dimension two, elliptic fibrations play an important role in the classification and study of algebraic surfaces. Consequently, explicitly constructing and describing birational models in a relative setting becomes highly relevant. 

In this paper, we provide two distinct birational models for elliptic surfaces equipped with a line bundle $\kl$ of relative degree two. We define an elliptic surface to be a normal surface $S$ equipped with a proper, flat, surjective morphism $f \colon S \to C$ onto a smooth curve $C$, whose generic fibre is a smooth curve of genus one. Furthermore, we assume that $f$ is relatively minimal; see Section \ref{sect: preparations} for details.

When $f$ admits a section, the classical Weierstrass model provides a well-studied birational model of $S$, embedding the surface as a relative cubic hypersurface in a $\IP^2$-bundle over the base curve $C$, and it is a cornerstone in the study of moduli problems for elliptic surfaces \cite{MirandaModuli, inchiostro2020, ModuliKodairaOne}. However, without assuming that a section exists, much less is known — particularly when multiple fibres are present —and explicitly describing convenient birational models with a view towards moduli problems is considerably more challenging. The two models introduced in this paper aim to bridge this gap in the degree-two case, advancing the study of elliptic surfaces with higher-degree relative polarisations.

The birational double cover model, which we describe in Section \ref{sect: birational double cover model}, had already been considered independently in several special cases \cite{horikawa, do-rollenske22, ST, MZ, inchiostro2026}. It starts from the simple observation that the datum of  $\kl$ realises the generic fibre as a double cover of $\IP^1$ (cf., Lemma \ref{lem: linear systems on Kodaira fibres}) and hence $S$, birationally, as a double cover over a ruled surface. We make this construction explicit and also work out the building data with a focus on applications. Analogues for relative polarisations of higher degree will be explored in a future project. 

The second model, which we call the relative determinantal model, relies on the study of the relative section ring of $\kl$ in the spirit of the Weierstrass model. In the presence of double fibres, we get a natural embedding into a weighted projective bundle that is locally expressed by determinantal equations (Theorem \ref{thm:embed_double}). To ensure that the equations can be globalised, as in the traditional Weierstrass model, we require an additional splitting condition. If, in addition, the base is $\IP^1$, this construction gives a model inside a toric $5$-fold. This is the content of Section \ref{sect:determinantal}. 

We also highlight the scope and limitations of our constructions through a series of examples in Section \ref{sect:examples}, including a complete characterisation of the determinantal model in the case of Halphen surfaces. In Section \ref{sect: Enriques} we discuss the models for special and non-special  Enriques surfaces and use them to reprove Horikawa's Theorem that all Enriques surfaces are deformation equivalent.

\subsection*{Acknowledgements}
S.R. and A.Z. are grateful to the organisers of the workshop ``Explicit Moduli Problems in Higher Dimensions" at Banff, where their collaboration started. The relative determinantal models were studied in A.U.'s Master's thesis, supervised by the first author and supported through the Erasmus program. We are indebted to Stephen Coughlan for the construction of the deformation in Section \ref{sect: Stephens family} and would like to thank Diana Torres and  Roberto Pignatelli for enriching discussions. 
S.R. was supported by the DFG.

\section{Generalities on elliptic surfaces, double covers of smooth surfaces and weighted projective bundles}\label{sect: preparations}

This section establishes the background material used throughout the paper. We work over the field of complex numbers $\mathbb{C}$.

\subsection{Elliptic surfaces}
Our starting point will always be a smooth elliptic surface, but since some of the models we produce are singular, we phrase the definition slightly more generally.

\begin{defin}\label{def:elliptic_fibration}
    Let $S$ be a normal surface. An \newnotion{elliptic fibration} on $S$ is a proper, flat, surjective morphism $f \colon S \to C$ onto a smooth curve such that the generic fibre $S_{\eta}$ is a smooth curve of genus one. The fibration is said to be \newnotion{relatively minimal} if $K_S$ is $f$-nef.
    
      An \newnotion{elliptic surface} is a normal surface $S$ equipped with an elliptic fibration $f \colon S \to C$.
\end{defin}

\begin{rem}
Note that we do not require a global section for $f$ in Definition \ref{def:elliptic_fibration}. This allows for multiple fibres. 
\end{rem}

Given an elliptic fibration $f \colon S \to C$ on a smooth surface $S$, repeatedly contracting the $(-1)$-curves contained in its fibres yields a relatively minimal fibration. This relatively minimal model is unique \cite[Propositions~III.8.4 and~III.8.5]{BHPV}. 
Unless stated otherwise, we will thus always assume that an elliptic fibration on a smooth surface is relatively minimal.

The fibration  $f$ has only finitely many singular fibres. The singular fibres were classified by Kodaira and N\'{e}ron \cite{kodaira1, kodaira2, neron}; their full classification is summarised in Table~\ref{tab:fibres}. These fibres are unions of potentially non-reduced rational curves whose dual graphs are extended Dynkin diagrams of ADE type. Furthermore, except for types $I_0$, $I_1$, and $II$, every irreducible component of a singular fibre is a $(-2)$-curve.

\begin{table}
\caption{Kodaira's classification of singular fibres}
\label{tab:fibres}
\begin{tabular}{lcc}
\toprule 
Kodaira type & number of components & dual graph\\
\midrule 
$I_0$ & 1 (smooth) & \begin{tikzpicture}[line cap=round,line join=round,>=triangle 45,x=1.0cm,y=1.0cm]
\clip(-.2,-.2) rectangle (.2,.2);
\begin{scriptsize}
\draw [fill=black] (0.,0.) circle (2.5pt);
\end{scriptsize}
\end{tikzpicture}\\
$I_1$ & 1 (with a node) & \begin{tikzpicture}[line cap=round,line join=round,>=triangle 45,x=1.0cm,y=1.0cm]
\clip(-.2,-.2) rectangle (.2,.2);
\begin{scriptsize}
\draw [fill=black] (0.,0.) circle (2.5pt);
\end{scriptsize}
\end{tikzpicture}\\
$I_k$ & $k\geq 2$ (in a cycle) & $\tilde{A}_{k-1}$\\
$II$ & 1 (with a cusp) & \begin{tikzpicture}[line cap=round,line join=round,>=triangle 45,x=1.0cm,y=1.0cm]
\clip(-.2,-.2) rectangle (.2,.2);
\begin{scriptsize}
\draw [fill=black] (0.,0.) circle (2.5pt);
\end{scriptsize}
\end{tikzpicture}\\
$III$ & 2 (tangent) & $\tilde{A}_1$ \\
$IV$ & 3 (meeting in one point) & $\tilde{A}_2$ \\
$I_k^*$ & $k+5$ & $\tilde{D}_{4+k}$\\
$IV^*$ & $7$ & $\tilde{E}_6$\\
$III^*$ & $8$ & $\tilde{E}_7$\\
$II^*$ & $9$ & $\tilde{E}_8$\\
\bottomrule
\end{tabular}
\end{table}

Whenever $f$ has a multiple fibre $mF$, the reduced fibre $F$ is of type $I_k$ for some $k \geq 0$ (see, e.g., \cite[Proposition~5.1.8]{dc} or \cite[Section~V.7]{BHPV}), and the sheaves $\ko_S(F)$ and $\ko_F(F)$ are non-trivial torsion line bundles of order $m$ (see, e.g., \cite[Lemma~III.8.3]{BHPV}). When $m=2$, we have the following result.

\begin{lem}\label{lem:torsion}
    Let $E = 2F$ be a double fibre of an elliptic fibration $f\colon  S \to C$, with $S$ a smooth projective surface. Then there is a short exact sequence
\[ 
0 \to \ko_F(-F) \to \ko_E \to \ko_F\to 0,
\]
and  we have a decomposition $\ko_{E} \isom  \ko_F \oplus  \ko_F(-F)$ as a sheaf on $F$.
\end{lem}

\begin{proof}
Since the bundle $\ko_F(-F)$ is a non-trivial $2$-torsion bundle on $F$, we have that $\Ext^1(\ko_F, \ko_F(-F)) = H^1(F, \ko_F(-F)) = 0$.
\end{proof}

\subsection{Double covers of smooth surfaces}\label{sect:double covers} 

For the reader's convenience, we recall some basic facts about (branched) double covers of smooth surfaces that will be used in Section~\ref{sect: birational double cover model}. A standard reference is \cite[Ch.~I.17]{BHPV}, or more generally \cite{Pardini91} for abelian covers.

A double cover of a smooth surface $W$ is a finite morphism $\pi \colon Y \to W$ of degree two. In such a situation, the trace map splits the pushforward of the structure sheaf as $\pi_* \ko_Y \isom \ko_W \oplus \km^{-1}$, where the algebra structure is determined by the multiplication map $\km^{-1} \tensor \km^{-1} \to \ko_W$. Its dual defines an isomorphism of line bundles $\km^{\otimes 2} \isom \ko_W(B)$, where $B \subset W$ is the branch divisor of $\pi$. Conversely, given a line bundle $\km$ on $W$ and a section $s_B \in H^0(W, \km^{\otimes 2})$, this  uniquely determines a double cover $\pi \colon Y \to W$. Letting $t \in H^0(|\km|, \pi^*\km)$ be the tautological section on the total space of the geometric line bundle $|\km| \to W$, the surface $Y$ is defined by the equation $t^2 = s_B$. By a standard abuse of terminology, we refer to $Y$ as the double cover of $W$ branched along $B$. By adjunction, its canonical bundle is $\omega_Y=\pi^*(\omega_W\otimes \km)$.

Since the base surface $W$ is smooth, the covering surface $Y$ is singular only over the singular points of the branch divisor $B$. Furthermore, the local analytic type of these singularities is uniquely determined by the local structure of $B$. In fact, assuming $Y$ is normal — which occurs if and only if $B$ is reduced — a resolution of $Y$ can be constructed via an embedded log resolution of the singularities of $B$, as we now explain.

\begin{construction}\label{construction: equivariant resolution}
If $B$ is reduced, then we can 
construct an embedded log resolution $\mu \colon \tilde W \to W$ of $B$ with the 
following properties:
\begin{enumerate}
 \item the strict transform $\inverse \mu _* B$ is smooth;
 \item decomposing the pullback $\mu^*B$ as $2\tilde A + \tilde B$ with $\tilde B$ reduced, the  divisor 
$\tilde B\geq\inverse \mu _* B$ is still smooth.
\end{enumerate}
In particular, we obtain the diagram below, where $\tilde{Y}$ is the double cover of $\tilde{W}$ branched over $\tilde{B}$. This yields an equivariant resolution of the singularities of $Y$ that is not necessarily minimal; however, a minimal resolution can be obtained by repeatedly contracting $(-1)$-curves.

\[
 \begin{tikzcd}\tilde Y \rar{\text{normal.}} 
\arrow{dr}[swap]{\text{branched over $\tilde B$}} & \mu^* Y 
\dar{\mu^*\pi}\rar & Y \dar{\pi}\\
& \tilde W \rar{\mu} &W
\end{tikzcd}
\]
\end{construction}
Here, by equivariant, we mean that the resolution is compatible with the lifted deck involution. In fact, it can also be obtained by repeated blow-ups along equivariant centres.

\begin{exam}\label{ex: singularity for double fibre}
We now explicitly apply Construction~\ref{construction: equivariant resolution} to the case where the branch divisor $B$ has an isolated triple point $p$ satisfying the following conditions:
\begin{enumerate}
\item $B = \Gamma + B_0$ in a neighborhood of $p$;
\item $\Gamma$ is smooth at $p$;
\item $B_0$ has an $A_{n+3}$ singularity at $p$ with $n \geq 0$; and
\item the intersection multiplicity is $(B_0 \cdot \Gamma)_p = 4$.
\end{enumerate}
In Arnold's notation \cite{Arnold2012}, $B$ has a singularity of type $J_{2,n}$ at $p$.

We determine the resulting singularity types of $Y$ over $p$ and outline the construction of their minimal resolutions. First, observe that by assumption, a single blow-up at $p$ yields a triple point $p'$ on the strict transform of $B$. Blowing up a second time at $p'$ separates the strict transforms of $\Gamma$ and $B_0$. Let $E_1$ and $E_2$ denote the exceptional curves of the first and second blow-ups, respectively. Following these two blow-ups, we denote the strict transforms of the components of $B$ by $\tilde{\Gamma}$ and $\tilde{B}_0$. The branch locus of the normalisation $\tilde{Y}$ of the pullback cover can then be identified as $\tilde{B} = E_1 + \tilde{\Gamma} + \tilde{B}_0$. Since $E_1^2 = -2$ on the base surface, it lifts to a $(-1)$-curve on the cover $\tilde{Y}$. Contracting this $(-1)$-curve yields a new surface, $\bar{Y}$. To determine the singularity type over $p$ and its minimal resolution, there are three cases to consider.

\begin{prooflist}
 \item[$n=0$] In this case, $\tilde B $ intersects $E_2$ in four distinct 
points, so in the surface $\bar Y$, $E_2$ becomes a smooth elliptic curve $F$ of 
self-intersection $(-1)$ and $\bar Y \to Y$ is the minimal 
resolution of a minimally elliptic singularity named $T_{2,3,6}$ or $J_{2,0}$ in Arnold's classification.
 \item[$n=1$] In this case, $\tilde{B}$ is still smooth but no longer intersects $E_2$ transversely; instead, the two curves are tangent at a single point. Consequently, the exceptional divisor $E_2$ lifts to a nodal rational curve on $\bar{Y}$, meaning that $Y$ has a cusp singularity of type $T_{2,3,7}$.
 
 \item[$n\geq2$]  In this case, the strict transform of $B_0$ still has an $A_{n-1}$ singularity at its intersection point with $E_2$. By either resolving this $A_{n-1}$ surface singularity on $\bar{Y}$ or proceeding as in Construction \ref{construction: equivariant resolution}, we find that the minimal resolution of $\bar{Y}$ is the minimal resolution of $Y$. The exceptional divisor consists of a cycle of $n$ rational curves, and all components have self-intersection $-2$, except for the strict transform of $E_2$, which has self-intersection $-3$. Consequently, the entire cycle has a total self-intersection of $-1$, and its contraction down to $Y$ yields a cusp singularity of type $T_{2,3,6+n}$.
\end{prooflist}
\end{exam}

\subsection{Weighted projective bundles}\label{sect:wpb}

We now explain how to construct weighted projective bundles over a smooth base $Z$ by taking a relative $\Proj$ of a sheaf of graded $\ko_Z$-algebras, following \cite[Section 4]{Mullet} and \cite[Section 3]{cp22}.
Subsequently, in Sections \ref{sect:determinantal}, \ref{sect:examples}, and \ref{sect: Enriques}, we consider the case where $Z=C$ is the base curve of an elliptic fibration.

If $\kf$ is a locally free sheaf on $Z$ and $w$ is a positive integer, we can consider the weighted symmetric algebra defined as $\Sym^{(w)}(\kf)=\bigoplus_{d\geq 0} \Sym^{(w)}(\kf)_d$, where 
\[
\Sym^{(w)}(\kf)_d = \begin{cases}
    \Sym(\kf)_k & \text{if}\, d=kw \\
    0 & \text{otherwise}
\end{cases},
\]
and $\Sym(\kf)_k$ denotes the $k$-th graded  piece of the usual symmetric algebra. Equivalently, if $\kf$ has rank $r$, we set $
\Sym^{(w)}(\kf)_0 \simeq \ko_Z$ and we cover $Z$ by opens such that the restriction of $\Sym^{(w)}(\kf)$ to each open is simply a polynomial ring, on some variables $s_1,\ldots,s_r$, graded by $\deg s_i=w$.  

With the above notation in mind, we can now introduce the following definition.

\begin{defin}
    Given positive integers $w_1\leq w_2 \leq \ldots \leq w_n$ and locally free sheaves $\kf_1,\ldots,\kf_n$ on $Z$, the \newnotion{associated free weighted symmetric algebra} is the sheaf of graded $\ko_Z$-algebras
    \[
    \ks  = \wSym_{(w_1,\ldots,w_n)}(\kf_1,\ldots,\kf_n) \coloneqq \Sym^{(w_1)}(\kf_1) \otimes_{\ko_Z}\ldots \otimes_{\ko_Z} \Sym^{(w_n)}(\kf_n)
    \]
    with weights $(w_1, \dots, w_n)$. 
   
   The corresponding \newnotion{weighted projective bundle} is
    \[ \pi \colon \mathbb{P}_{(w_1,\ldots,w_n)}(\kf_1,\ldots,\kf_n)\coloneq \Proj_Z(\ks) \to Z.\]
\end{defin}
Denoting the weighted projective bundle by $\pi \colon \IP\to Z$ momentarily, the fibres of $\pi$ are isomorphic to the weighted projective space $\IP_z \isom \mathbb{P}(w_1^{\rk \kf_1},\ldots,w_n^{\rk \kf_n})$. 
 Moreover, there exist sheaves $\ko_{\mathbb{P}}(d)$ for all $d \in \mathbb{Z}$ whose restriction to each fibre of $\pi$ coincides with $\ko_{\mathbb{P}_z}(d)$.

If each $\kf_i$  splits as a direct sum of line bundles,  $\kf_i = \bigoplus_{j=1}^{\rk \kf_i} \kl_{ij}$, then there are canonical sections 
\[
x_{ij}\in H^0\left( Z, \mathcal{L}_{ij}^{-1} \otimes \pi_* \ko_{\mathbb{P}}(w_i) \right),
\]
which we call the generator, global coordinate or just variable associated to $\kl_{ij}$. We can then write $\ks=\bigotimes_{ij} \bigoplus_k \kl_{ij}^k x_{ij}^k$ and indeed the $x_{ij}$ induce (weighted) homogeneous coordinates on each fibre. 
 If in addition $Z$ is a projective space, then $\IP$ is the toric variety in \cite[Construction 3.2]{Mullet}.

Let us illustrate this in an example that will reappear in  Remark \ref{rem: numbers for deg 2}.
\begin{exam}\label{example: toric threefold}
	Let $Z=\IP^1$ and consider, for some integers $d_1$, $d_2$, $d_3$, the line bundles $\kf_1=\ko_{\IP^1}(-d_1)$, $\kf_2=\ko_{\IP^1}(-d_2)$ and $\kf_3=\ko_{\IP^1}(-d_3)$. Then the 
	weighted projective bundle $\IP_{(w_1, w_2, w_3)}(\kf_1,\kf_2,\kf_3)$ over $\IP^1$, is the toric variety $\IC^5//(\IC^{\times})^2$ with weight matrix
    \[
\mat{t_0 & t_1 & x_0 & x_1 & y \\
1 & 1 & d_1 & d_2 & d_3 \\
0 & 0 & w_1 & w_2 & w_3}
    \]
    and irrelevant ideal $I=(t_0,t_1)\cap (x_0,x_1,y)$. That is, the quotient of $\mathbb{C}^5\backslash V(I)$ by the action of $(\IC^{\times})^2$ given by 
\[
(\lambda,\mu)\cdot (t_0,t_1,x_0,x_1,y)=(\lambda t_0,\lambda t_1, \lambda^{d_1}\mu^{w_1} x_0, \lambda^{d_2}\mu^{w_2}x_1, \lambda^{d_3}\mu^{w_3}y ).
\]

    Note that in this case, writing $\IP^1 = \Proj \IC[t_0,t_1]$ and letting $U_i= \spec \IC[t_0/t_i,t_1/t_i]$, for $i=0,1$, we have that 
    \[
    \ks|_{U_i}\simeq \IC\left[\frac{t_0}{t_i},\frac{t_1}{t_i}\right][t_i^{-d_1}x_0,t_i^{-d_2}x_1,t_i^{-d_3}y]\simeq \ko_{U_i}[s_1,s_2,s_3] 
    \]
    with $\deg s_i=w_i$ and hence $\Proj_Z(S)|_{U_i}\simeq \IA^1 \times \IP(w_1,w_2,w_3)$. Moreover, we have $x_0$ as a generator of $\kf_1$, $x_1$ for $\kf_2$, and $y$ for $\kf_3$; these are precisely the aforementioned global coordinates on $\Proj_Z(\ks)$ in this case.
\end{exam}

\section{Marked elliptic surfaces and the ample model}\label{sect:marked_and_ample}

As noted in the introduction, this paper aims to provide explicit birational models for elliptic surfaces equipped with a line bundle $\kl$ of relative degree two. To formalise this objective, we define the following objects.

\begin{defin}\label{def:markedpair} 
    A \newnotion{marked elliptic surface of relative degree two}, or simply a marked elliptic surface, is a pair $(f\colon S\to C, \kl)$ consisting of:
    \begin{enumerate}
        \item a relatively minimal elliptic fibration $f\colon S \to C$, and
        \item a choice of $f$-nef line bundle $\kl$ on $S$, whose restriction $\kl_c$ to any fibre $S_c$ of $f$ has degree two.     
    \end{enumerate} 
\end{defin}

Given a marked elliptic surface $(f\colon S\to C, \kl)$, we will refer to $\kl$ as either a marking or a relative polarisation. We denote the reductions of the double fibres of $f$, if any, by $F_1, \ldots, F_k$. The corresponding points on $C$ over which these double fibres lie are denoted by $c_i$.

In this setup, a natural birational model to consider is the relative $\Proj$ of the relative section ring $\kr(f,\kl)$ associated with the marking $\kl$, which is a sheaf of graded $\ko_C$-algebras. To simplify terminology and notation, we introduce the following definition.

\begin{defin}\label{def:ample_model}
The \newnotion{ample model} of a smooth marked elliptic surface of relative degree two $(f\colon S\to C,\kl)$  is the pair $(\tilde f\colon X\to C,\tilde \kl)$ defined by 
\[
X \coloneqq \Proj_C\kr(f,\kl)=\Proj_C\left(\bigoplus_{n\geq 0} f_*\kl^{\tensor n} \right)
\]
and the relation $\kl = \phi^*\tilde \kl$ for the natural map $\phi \colon S \to X$. 
\end{defin}

In the subsequent sections, we will geometrically describe the ample model $X$ of a marked elliptic surface $(f\colon S \to C, \kl)$ from two different points of view. A few remarks concerning such a model are thus in order.

\begin{rem} 
First, it is routine to check, for example, using Reider's theorem \cite[Chapter IV, Theorem 11.4]{BHPV}, that a multiple of $\kl$ tensored with a sufficiently positive line bundle from the base curve is globally generated; in particular, $\phi \colon S \to X$ is a morphism.
We also observe that a curve contained in a fibre of $f$ is contracted by $\phi$ if and only if the restriction of $\kl$ to this curve has degree zero. Since $f \colon S \to C$ is relatively minimal, by the classification of singular fibres (Table \ref{tab:fibres}),  any strict subconfiguration of a reducible fibre contracts to an ADE singularity, and thus $X$ has canonical singularities. 
 \end{rem}

\begin{rem}\label{rem: bisection}
Second, since we are working relative to the base curve $C$, we are often willing to replace $\kl$ with $\kl\tensor f^*\kl'$ for a suitable line bundle $\kl' \in \Pic(C)$. By doing so, we can always assume that $\kl = \ko_S(D)$ for an effective divisor $D$. Sometimes it is convenient to consider $\kl=\omega_{S/C}(D)$ instead. Both relative polarisations $\ko_S(D)$ and $\omega_{S/C}(D)$ yield the same ample model (in a slightly different embedding) since, by the canonical bundle formula \eqref{eq: can bundle formula}, $\omega_{S/C}^{\tensor 2}$ is the pullback of a bundle on $C$. For a concrete example, see Section~\ref{sec:Halphen}.
\end{rem}

\begin{rem}\label{rem:rel_lc}
    Third, we note that if $\kl = \ko_S(D)$ is defined by an effective divisor and the pair $(S, D)$ is log canonical, then the ample model is isomorphic to the so-called relative canonical model, see e.g., \cite[Section 3.8]{komo}.
\end{rem}

For later reference, we include next the formulas computing the main invariants of $X$ in terms of the so-called fundamental line bundle $\tilde f_*\omega_{X/C} = R^1\tilde f_* \ko_X$. These formulas are usually stated for the smooth surface $S$, but since $X$ has canonical singularities, they are equally valid for $X$. On the ample model, we denote the double fibres by $2G_1, \ldots, 2G_k$, that is, for each $i$, we set $G_i\coloneqq \tilde{f}^{-1}(c_i)=\varphi(F_i)$.

\begin{prop}\label{prop: numerical invariants}
	With notations as in Definition \ref{def:ample_model}, we have that:
	\begin{align}
		\omega_{X/C} & = 
		\tilde f^*\tilde f_*\omega_{X/C} \left(\tsum_{i=1}^k G_i\right);
		 \label{eq: can bundle formula}
		\\
		p_g(X)& = \begin{cases}
			g(C) &   \text{if $\tilde f_*\omega_{X/C}$ is trivial, or} \\
			\deg \tilde f_*\omega_{X/C}+g(C) -1 & \text{otherwise}
		\end{cases} \,\,\text{;}\\
		q(X)& = \begin{cases}
			g(C)+1 &    \text{if $\tilde f_*\omega_{X/C}$ is trivial, or} \\
			g(C) & \text{otherwise}
		\end{cases}.
	\end{align}
	
\end{prop}
\begin{proof}
	The formulas are 
	given in \cite[Sect. 7, Lemma 14]{Friedman} or \cite[Chapter V, Sect. 12]{BHPV}
\end{proof}

For each $n \geq 0$, we further let $\ke_n \coloneqq f_* \kl^{\tensor n}$ and record the following standard auxiliary result that follows, for example, from  \cite[Chapter I, Theorem 8.5 (iv)]{BHPV}.

\begin{lem}
    For each $n\geq 0$, the sheaf $\ke_n$ is a vector bundle of rank $2n$. Moreover, the fibre over a point $c\in C$ is naturally isomorphic  to $H^0 (S_c, \kl_c^{\otimes n})$.
\end{lem}

Lastly, we observe that since $f\colon S\to C$ is relatively minimal, the line bundle $\kl$ induces an involution on the generic fibre $S_{\eta}$ that extends to a biregular involution on both $S$ and $X$. In particular, we can decompose each locally free sheaf $\ke_n$ into a direct sum of its invariant and anti-invariant parts: $\ke_n = \ke_n^{+} \oplus \ke_n^{-}$. Furthermore, the rational map $\varphi$ realises $S$ \textbf{birationally} as a double cover of the ruled surface $W \coloneqq \IP(\ke_1)$, whose branch divisor $B$ satisfies $B \cdot \Gamma = 4$ for a fibre $\Gamma$ of $W \to C$. We summarise this framework in the following diagram, where $\tau$ (resp. $\tilde \tau$) denotes the aforementioned involution on $S$ (resp. $X$).

\begin{figure}[h]
\centering
\begin{tikzcd}[row sep=1.5em, column sep=3 em]
S \arrow[r, "\varphi"] \arrow[d, "f"'] \arrow[loop above, "\tau", distance=2em, in=45, out=125]
  & X \arrow[dashed, d, "\theta"] \arrow[dl, "\tilde f" description]
    \arrow[loop above, "\tilde \tau", distance=2em, in=45, out=125] \\
C 
  & \quad W \supseteq B \arrow[l] \\
\end{tikzcd}
\label{fig:double-cover}
\end{figure}

Making this (birational) double cover description explicit will be the content of the next section, while in Section \ref{sect:determinantal}, we will explicitly describe $X$ as embedded in a weighted projective bundle, a description that stems from a careful local analysis of the relative section ring, which we do in Section \ref{sect: local models}. We call the latter the determinantal model. 

From the perspective of moduli theory, the ample model $X$ is a natural birational model to consider, aligning with modern approaches to compactifying moduli spaces of elliptic surfaces as in \cite{KD_moduli}, where the authors consider the case of relative degree one; see also Remarks \ref{rem: bisection} and \ref{rem:rel_lc}. Depending on the specific class of surfaces being parameterised, one may choose between the birational double-cover and determinantal descriptions. Concrete examples of marked elliptic surfaces as in Definition \ref{def:markedpair}, which are of interest to us, arise from an elliptic fibration admitting double fibres, a bisection, or both. Two important classes are Halphen surfaces of index two, discussed in Section \ref{sec:Halphen}, and Enriques surfaces, treated in Section \ref{sect: Enriques}.

\section{Birational double cover models}
\label{sect: birational double cover model}

Let  $(\tilde f\colon X\to C, \tilde \kl)$ be the ample model of a marked elliptic surface of relative degree two $(f \colon S \to C, \kl)$, which we henceforth fix. 
The purpose of this section is to make the description of the birational double cover $\theta\colon X \dashrightarrow W$ explicit, focusing on the case where there is at least one double fibre, and to give a concrete recipe for constructing elliptic surfaces with double fibres beyond Halphen and Enriques surfaces. 

We continue to use the notation from Section \ref{sect:marked_and_ample} and will need some lemmata that are well known and follow from, e.g., \cite{catanese-franciosi96} or \cite{CFHR}. The main ingredient is that the 
Kodaira fibre has arithmetic genus one, but every proper subcurve has arithmetic 
genus zero. An alternative argument is given by the local analysis in Section \ref{sect: local models}.

First, for the double fibres, $2F_1,\ldots,2F_k$, the relative linear systems $|\kl|_{F_i}|$ verify the following.

\begin{lem}\label{lem: simple_base_pts}
    For each $F_i$, the linear system $|\kl|_{F_i}|$ has a unique base point $q_i$, which is a simple base point.
\end{lem}

For the non-multiple fibres, we have the following more general and standard geometric characterisation.

\begin{lem}\label{lem: linear systems on Kodaira fibres}
 Let $F$ be a (non-multiple) Kodaira fibre and $D$ a Cartier divisor on $F$,  which 
has non-negative degree on every component of $F$. If $\deg_F D=2$, then $h^0(F, D) = 2$, $h^1(F, D) = 0 $ and  $|D|$ defines a morphism 
 \[
 \begin{tikzcd}
  F \rar & \bar F \rar & \IP^1
 \end{tikzcd},
 \]
where the first arrow contracts the components of $F$ on which $D$ has degree 
$0$ and the second is a double cover branched over four (not necessarily 
distinct) points.
\end{lem}

With these in mind, we can now introduce what we call the birational double cover model of $f \colon S \to C$ by proving the following.

\begin{thm}\label{thm: birational double cover model}
 There is a commutative diagram
 
\begin{equation}\label{diagram: model}
	\begin{tikzcd}
		& & \hat X \arrow{dl}[swap]{ \sigma} \arrow[bend left]{drr}{\hat \theta} \arrow{ddl}{\hat f}
		\arrow{dr}{\bar \sigma} \\
		S \rar{\varphi} \arrow{dr}[swap]{f} & X \dar[swap]{\tilde f} && \bar X 
		\dar[swap]{\bar f} \rar{\bar \theta} 
		&W=\IP(\ke_1)\arrow{dl}{\alpha}\\
		& C \arrow[equal]{rr} && C
	\end{tikzcd}.
\end{equation}

where:  
\begin{enumerate}
 \item $\sigma$ is the blow up at the $k$ points $\varphi(q_i)$ with total exceptional divisor $\hat E = \sum_i \hat E_i$ and $\hat f$ is the induced fibration;
 \item $\hat \theta$ is the morphism induced by the surjection 
 \[\hat f^* \ke_1 \isom  \hat f^*\hat f_* \left(\sigma^* \tilde \kl(-\hat E)\right) \onto \sigma^*\tilde \kl(-\hat E);\]
 \item$\bar \sigma$ is the Stein factorisation of $\hat \theta$ and given by the contraction of the elliptic curves   $\hat G_i = \inverse{\sigma}_* (G_i)$;
 \item $\bar \theta$ is a double cover with branch divisor of the form $B = B_0 + \sum_i \Gamma_i$, where $\Gamma_i$ is the fibre of $\alpha: W \to C$ over $c_i$, and associated line bundle (see Section \ref{sect:double covers})  $\km=\alpha^*\kn \tensor \ko_W(2)$, for some $\kn \in \Pic(C)$. 
\item  At the points $p_i = \bar \theta(\hat G_i)$, the images of the elliptic singularities of $\bar X$, the divisor $B$ has a singularity of type $J_{2,n}$, which means that $B_0$ has a singularity of type $A_{n+3}$. 
 \end{enumerate}
\end{thm}

Then $\theta = \hat \theta \circ \inverse \sigma$ and we call $\bar \theta \colon \bar X \to W$ the \newnotion{birational double cover model} of $(S, \kl)$. Note that on $\bar{X}$, the induced relative polarisation is simply $\theta^*\ko_W(1)$.

\begin{proof}
We first observe that surjectivity in $(ii)$ follows from Lemmas \ref{lem: simple_base_pts} and \ref{lem: linear systems on Kodaira fibres}. 

Next, we establish the isomorphism $\hat f_* \hat \theta^* \ko_W(1)=\hat{f}_*\left(\sigma^* \tilde{\kl}( - \hat{E})\right) \cong \ke_1$. Consider the natural restriction sequence on $\hat{X}$ associated with the total exceptional divisor $\hat{E} = \sum_{i=1}^k \hat{E}_i$:
\[ 
\begin{tikzcd}
0 \rar & \sigma^*\tilde{\kl}(-\hat{E}) \rar & \sigma^*\tilde{\kl} \rar & \bigoplus_{i=1}^k \ko_{\hat{E}_i} \rar & 0.  
\end{tikzcd}
\]
Taking the direct image under the fibration $\hat{f}_*$ yields the long exact sequence of sheaves on the base curve $C$:
\[
\begin{tikzcd}[row sep = small]
0 \rar & \hat{f}_* \left(\sigma^*\tilde{\kl}(-\hat{E})\right) 
\rar & \hat{f}_* \left(\sigma^*\tilde{\kl}\right) \rar{\delta} & \bigoplus_{i=1}^k \ko_{c_i} \rar &{}\\
{} \rar & R^1\hat{f}_* \left(\sigma^*\tilde{\kl}(-\hat{E})\right) \rar & R^1\hat{f}_* \left(\sigma^*\tilde{\kl}\right) \rar & 0, 
\end{tikzcd}
\]
where, by the projection formula, $R^1\hat{f}_* \left(\sigma^*\tilde{\kl}\right) \cong R^1f_*\kl = 0$ since $\kl$ restricted to any fibre has no higher cohomology by our assumptions. 

Because the points $\varphi(q_i)$ are the unique base points of the relative linear systems on the double fibres, every local section of $\hat{f}_* \left(\sigma^*\tilde{\kl}\right)$ vanishes along the exceptional curves $\hat{E}_i$. Consequently, the evaluation map $\delta$ is identically the zero morphism, yielding the isomorphism:
\[ 
\hat{f}_* \left(\sigma^*\tilde{\kl}(-\hat{E})\right) \cong \hat{f}_* \left(\sigma^*\tilde{\kl}\right) \cong f_*\kl = \ke_1. 
\]
Furthermore, the vanishing of $\delta$ combined with the exactness of the remaining terms provides an identification 
\[ 
\textstyle\bigoplus_{i=1}^k \ko_{c_i} \isom R^1\hat{f}_* \left(\sigma^*\tilde{\kl}(-\hat{E})\right). 
\]

Now, the map $\hat{\theta}$ is generically of degree two. Taking its Stein factorisation yields $\bar{\sigma}\colon \hat{X} \to \bar{X}$ and $\bar{\theta} \colon \bar{X}\to W$, where $\bar{\theta}$ is a double cover, and $\bar{\sigma}$ contracts precisely those curves on which $\sigma^*\tilde{\kl} (-\hat{E})$ has degree zero. Since $\tilde{\kl}$ is relatively ample on $X$ and $\deg_{\hat{E}_i} (\sigma^*\tilde{\kl} (-\hat{E})) = -(\hat{E}_i^2) = 1$, the exceptional curves are not contracted, while the strict transforms of the double fibres are contracted since $\deg_{\hat{G}_i}(\sigma^*\tilde{\kl} (-\hat{E})) = 0$. This proves $(iii)$.

For the statement $(iv)$, we track the scheme-theoretic fibre over $c_i$. We have $\hat{f}^{-1} (c_i) = 2(\hat{E}_i + \hat{G}_i)$, and thus on $\bar{X}$, $\bar{f}^{-1}(c_i) = 2 \bar{E}_i$. Now, because the ruled surface $W$ contains no multiple fibres, each fibre $\Gamma_i$ over $c_i$ must be entirely contained in the branch locus of $\bar{\theta}$. 

Finally, because $\hat{G}_i^2 = -1$, the resulting elliptic singularity on $\bar{X}$ is, in Arnold's classification, a singularity of type $T_{2,3, 6+n} = J_{2,n}$. It follows that the branch curve $B$ must possess a singularity of the same type at $p_i = \bar{\theta}(\hat{G}_i)$, which implies that the remaining component $B_0$ meets $\Gamma_i$ with an $A_{n+3}$ singularity (cf., Example \ref{ex: singularity for double fibre}).
\end{proof}

\begin{rem}\label{rem: birational model in wp bundle}
	As is standard for double covers of projective bundles, we can naturally embed the birational double cover model $\bar X$ as a relative hypersurface of degree four in a weighted $\IP(1,1,2)$-bundle (see Section \ref{sect:wpb} for notation), namely,
	\[
	\begin{tikzcd}
		\bar X \rar[hookrightarrow] \dar{\bar f}\arrow{dr}{\bar\theta} & \IP_{(1,2)}(\ke_1, \kn^\vee)\dar[dashed]\\
		C & \IP(\ke_1)\lar[swap]{\alpha}
	\end{tikzcd}.
	\]
\end{rem}

\begin{lem}\label{lem: can bundle birat model}
With notations as in Theorem \ref{thm: birational double cover model}, we have an isomorphism of line bundles 
 $\ke_2^-\isom \kn^\vee(\sum_i c_i)$. In particular, the relative canonical bundle on the ample model is given by $\omega_{X/C} = \tilde f^*\left(\det \ke_1\tensor \kn\right) \left(-\tsum_{i=1}^k G_i\right)$, and its pushforward on $C$ can be expressed as
\[
\tilde f_*\omega_{X/C} = \det \ke_1\tensor \left(\ke_2^-\right)^\vee = \det \ke_1\tensor \kn \left(-\tsum_{i=1}^k c_i\right).
\]
 \end{lem}

\begin{proof}
We first establish the isomorphism  $\ke_2^- \isom \kn^\vee\left(\sum_i c_i\right)$. For this, let $\Gamma \coloneqq \sum_{i=1}^k \Gamma_i$ denote the sum of the fibres of $W \to C$ over the points $c_i$, and, similarly, let $\hat G \coloneqq \sum_{i=1}^k \hat G_i$. Then
\begin{align*}
    \sigma^* \tilde\kl^{\tensor 2} & = \hat\theta^* \ko_W(2) (2\hat E) \\
    & = \hat\theta^* \ko_W(2) (2\bar \sigma^* \bar E - 2\hat G) \\
    & = \hat\theta^*\left( \ko_W(2) (\Gamma) \right)(-2\hat G).
\end{align*}
Applying $\hat f_*$ and taking the anti-invariant component yields:
\begin{align*}
    \ke_2^- & = \left(f_* \kl^{\tensor 2}\right)^- \\
    & = \left(\hat f_* \sigma^*\tilde \kl^{\tensor 2}\right)^- \\
    &= \alpha_* \left( \ko_W(2) (\Gamma) \tensor \left(\hat \theta_* \ko_{\hat X} (-2\hat G)\right) ^- \right).
    \end{align*}

Now, consider the ideal sheaf exact sequence associated with $2\hat G$:
    \begin{align*}
    0 & \longrightarrow \ko_{\hat X} (-2\hat G) \longrightarrow \ko_{\hat X} \longrightarrow \ko_{2\hat G} \longrightarrow 0.
\end{align*}
By looking at the local structure of each elliptic singularity as a hypersurface inside a weighted projective space, see \cite[Chapter 4]{reid97}, we deduce that the structure sheaf $\ko_{2\hat G}\simeq \bigoplus_{i=1}^k \ko_{2\hat G_i}$ is purely invariant under the action of the covering involution. Consequently, $\left(\hat \theta_* \ko_{\hat X} (-2\hat G)\right)^- \isom \bigl(\hat \theta_* \ko_{\hat X}\bigr)^-$. The calculation then proceeds as follows.
\begin{align*}
    \ke_2^- & = \alpha_* \left( \ko_W(2) (\Gamma) \tensor \bigl(\hat \theta_* \ko_{\hat X} \bigr) ^- \right)  
    = \alpha_* \left( \ko_W(2) (\Gamma) \tensor \bigl(\bar \theta_* \ko_{\bar X} \bigr) ^- \right) \\
    & = \alpha_* \left( \ko_W(2) (\Gamma) \tensor \left(\alpha^*\kn^\vee(-2)\right) \right)  
    = \alpha_* \left( \alpha^*\kn^\vee (\Gamma) \right) 
    = \kn^\vee\left(\tsum_{i=1}^k c_i\right).
\end{align*}

We now turn to verifying the relative canonical bundle formulas. By the standard Hurwitz formula for finite double covers applied to $\bar \theta \colon \bar X \to W$, we find:
\begin{align*}
 \omega_{\bar X/C} &= \bar \theta^* \left(\omega_{W/C}\tensor \alpha^*\kn (2) \right) \\
 & = \bar \theta^* \left( \alpha^*\det (\ke_1) (-2)\tensor \alpha^*\kn(2) \right) \\
 & = \bar f^* \left(\kn \tensor \det \ke_1\right).
\end{align*}
Since $\sigma$ is a blow-up at smooth points and $\bar \sigma$ resolves $k$ simple elliptic singularities, we have
\[ 
\omega_{\hat X/C} = \sigma^*\omega_{X/C} (\hat E) = \bar \sigma^*\omega_{\bar X/C}( -\hat G). 
\]
Applying $\sigma_*$ and using the projection formula  gives
\begin{align*}
  \omega_{X/C} & = \sigma_* \sigma^*\omega_{X/C} \\
 & = \sigma_* \left( \bar \sigma^* \omega_{\bar X/C} \left( - \hat E -\hat G\right) \right) \\
 & = \sigma_*\left(\hat f^* \left(\kn \tensor \det \ke_1\right) \left( -\hat E - \hat G\right)\right).
 \end{align*}
 Now, since for each $i$ we have $\sigma^* G_i = \hat{G}_i + \hat{E}_i$,
\begin{align*}
 \omega_{X/C} & = \sigma_*\sigma^*\left( \tilde f^* \left(\kn \tensor \det \ke_1\right)\left(-\tsum_{i=1}^k G_i\right)\right) = \tilde f^*\left(\kn \tensor \det \ke_1\right) \left(-\tsum_{i=1}^k G_i\right).
\end{align*}

Finally, pushing forward to the base curve $C$ via $\tilde f_*$ and using the projection formula again yields:
\begin{align*}
  \tilde f_*\omega_{X/C} = \left(\kn \tensor \det \ke_1\right)\tensor \tilde f_*\ko_{X} \left( -\tsum_{i=1}^k G_i\right).
\end{align*}

The identity $\tilde f_*\ko_{X} \left( -\sum_i G_i\right) = \ko_C(-\sum_i c_i)$ then follows from pushing forward the short exact sequence
\[ 
0 \longrightarrow \ko_X(-2\tsum_{i=1}^k G_i) \longrightarrow \ko_{X} \left( -\tsum_{i=1}^k G_i\right) \longrightarrow \bigoplus_{i=1}^k \ko_{G_i}(-G_i) \longrightarrow 0. 
\]
Observing that $\ko_X(-2\sum_i G_i) \isom \tilde f^*\ko_C(-\sum_i c_i)$, and noting that $H^0(G_i, \ko_{G_i}(-G_i)) = 0$ because $\ko_{G_i}(-G_i)$ is a non-trivial torsion line bundle, the higher terms vanish. 
\end{proof}

If we want to use this framework to construct elliptic fibrations with double fibres, we need to find a branch divisor $B \subset W$ with the precise singularity types required to produce the desired elliptic singularities on the double cover $\bar X$, which then pull back to the double fibres on the smooth model $S$. We spell this out in a particular case that frequently arises in applications. We assume that the elliptic surface $f \colon S \to C$ admits a bisection $D \subset S$ and construct the ample model using the relative polarisation $\kl = \ko_S(\sum F_i+D)$ (cf. Remark \ref{rem: bisection}). This choice has the advantage that the singular points $p_i$ of the branch divisor $B$ will all lie on a uniquely determined section of the projective bundle $W \to C$, as we will demonstrate below.

\begin{construction}\label{construction deg 2}
 Let $C$ be a smooth curve and assume that we are given the following data:
 \begin{enumerate}
  \item distinct points $c_1, \dots, c_k\in C$ with $k\geq 1$,
  \item a rank two vector bundle $\ke$ on $C$ with a section giving a short exact sequence $0 \to \ko_C \to \ke \to \det \ke \to 0$ and  defining a ruled surface
  $\alpha \colon  W = \IP(\ke) \to C$ with a section $\Sigma$,
  \item a line bunde $\kn \in \Pic (C)$,
  \item an effective divisor
  \[B_0\in \left|\alpha^*\kn^{\tensor 2}\left( 
4\Sigma-\textstyle\sum_i \Gamma_i\right) \right|\]
  such that with $\Gamma_i = \inverse \alpha(c_i)$ we have 
 \begin{enumerate}
 \item $B_0$ has an $A_{n_i+3}$ singularity at $\Gamma_i \cap \Sigma = \{ 
p_i\}$ with $n_i\geq 0$,
 \item the local intersection multiplicity satisfies $(B_0.\Gamma_i)_{p_i} = 4$, or equivalently $B_0\cap \Gamma_i = \{p_i \}$,
 \item $B_0$ has at most ADE singularities elsewhere. 
 \end{enumerate}

 \end{enumerate}

From this data, we will now construct an elliptic surface $f \colon S \to C$ with $S$ smooth, carrying a distinguished bisection $D$ and having double fibres $2F_1, \ldots, 2F_k$ over the points $c_1, \ldots, c_k$. The case with $k=1$ is illustrated in Figure \ref{fig: picture E12}, which was adapted from \cite{ST}. 

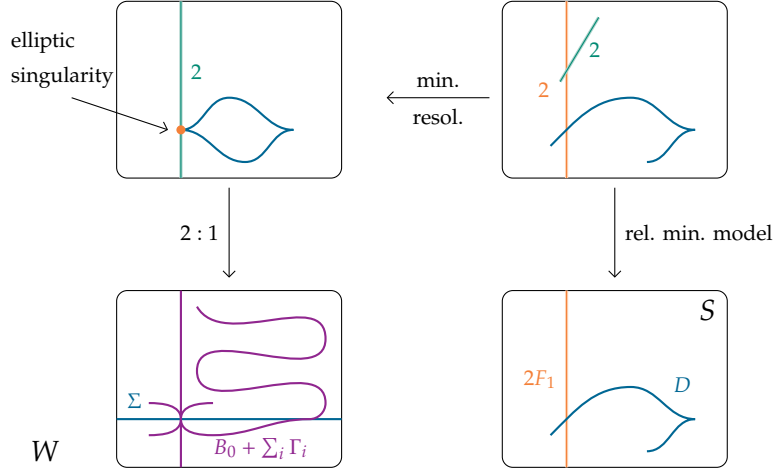
\begin{figure}[ht]
	\caption{A pictorial description when $k=1$}\label{fig: picture E12}
	\begin{tikzpicture}[scale = .85]
		
		\begin{scope}[xshift = 0cm, yshift = 4.5cm] 
			\draw[rounded corners] (-2.0, 1.5) rectangle (1.5, -1.25);
			\node[] at (1.2, 1.2) {};
			
			\draw[2G] (-1, -1.25) to ++ (0, 2.75);
			\draw[G] (-1, -1.25) to node[above right] {$2$} ++ (0, 2.75);
			
			\draw[E1] (-1, -.5) to[out = 0, in = 180] ++(.75, .5) to[out = 0, in = 180] node[above right] {} ++(1,- .5) to[out = 180, in = 0] ++(-.75,- .5) to[out = 180, in = 0] ++(-1,.5);
			
			\draw[pfeil] (-0.25, -1.4) -> node[right] {} node[left]{$2:1$} ++ (0, -1.4); 
			
			\fill[F] (-1, -.5) circle (2pt);
			
			\draw[<-] (-1.2, -.5) -- ++(-1.5, .5) node[above, text width=1.6cm] {\scriptsize elliptic singularity\,\,\,};
		\end{scope}

		\begin{scope}[xshift = 6cm, yshift = 4.5cm]
			\draw[rounded corners] (-2.0, 1.5) rectangle (1.5, -1.25);
			\node[] at (1.2, 1.2) {};
			
			\draw[2F] (-1, -1.25) to ++ (0, 2.75);
			\draw[F] (-1, -1.25) to node[left] {$2$\,\,} ++ (0, 2.75);
			\draw[2G] (-1.1, .25) to ++ (.6, 1);
			\draw[G] (-1.1, .25) to node[right] {$2$} ++ (.6, 1);
			
			\draw[E1] (-1, -.5) ++ (-.25, -.25) to[out = 45, in = 180] ++(1.25, .75) to[out = 0, in = 180]node[above right] {} ++(1,- .5) to[out = 180, in = 0] ++(-.75,- .5);
			
			\draw[pfeil] (-2.2, 0) -> node[below] {resol.} node[above]{min.} ++ (-1.6, 0); 
			
			\draw[pfeil] (-0.25, -1.4) -> node[right] {rel. min. model}  ++ (0, -1.4); 
		\end{scope}

		\begin{scope}[xshift = 0cm, yshift = 0cm]
			\draw[rounded corners] (-2.0, 1.5) rectangle (1.5, -1.25);
			\node[] at (-3.1, -1.0) {$W$};
			
			\draw[E1] (-2, -.5) node[above right] {$\Sigma$} to ++(3.5, 0);
			
			\draw[branch] (-1, -1.25) to ++ (0, 2.75);
			\draw [branch] (-1.5, -.25) to[out= 0, in =90] ++(.5, -.25) to[in = 0,out = -90] ++(-.5, -.25);
			\draw [branch] (-.5, -.25) to[out= 180, in =90] ++(-.5, -.25) to[in = 180,out = -90] ++(.5, -.25) to[out = 0 , in = 180] node[below] {$B_0+\sum_i \Gamma_i$} ++ (1.5, .25) to[out= 0, in =-90] ++(.25, .25) to[in = -90,out = 90] ++(-2, .5) to[in = -90,out = 90] ++(2, .5) to[in = -60,out = 90] ++(-2, .5);
		\end{scope}

		\begin{scope}[xshift = 6cm, yshift = 0cm]
			\draw[rounded corners] (-2.0, 1.5) rectangle (1.5, -1.25);
			\node[] at (1.2, 1.2) {$S$};
			
			\draw[2F] (-1, -1.25) to ++ (0, 2.75);
			\draw[F] (-1, -1.25) to node[left] {$2F_1$}++ (0, 2.75);
			
			\draw[E1] (-1, -.5) ++ (-.25, -.25) to[out = 45, in = 180] ++(1.25, .75) to[out = 0, in = 180] node[above right] {$D$}++(1,- .5) to[out = 180, in = 0] ++(-.75,- .5);
		\end{scope}
		
	\end{tikzpicture}
\end{figure}
Let $B\coloneqq B_0 + \sum_i \Gamma_i$ and  let $\bar 
\theta \colon \bar X \to  W$ be the double cover branched over $B$ defined 
by the isomorphism $\ko_{ W}(B) \isom \left( \alpha^*\kn (2\Sigma) 
\right)^{\tensor 2}$ as explained in Section \ref{sect:double covers}. We remark that the notation here is meant to be consistent with the one in Theorem \ref{thm: birational double cover model}.

Then the construction of $S$ goes as follows: we consider $\mu_1  \colon \hat S \to \bar X$ the minimal resolution and 
$\mu_2 \colon \hat S \to S$ the relative minimal model over $C$, resulting in 
a diagram
\begin{equation}\label{main diagram deg 2}
\begin{tikzcd}
  \bar X \dar[swap]{\bar\theta} & \hat S \lar[swap]{\mu_1} 
\arrow{dr}{\mu_2} \arrow{dl}{\psi}\\
   W \dar && S  \arrow[dashed]{ll}[swap]{} \dar{f} \\
   C \arrow[equal]{rr} &&C
   \end{tikzcd}.
  \end{equation}
Then $S$ is a smooth surface, and the induced morphism $f\colon S \to C$ is a relatively minimal elliptic fibration by construction. Moreover, by Lemma \ref{lem: 
double fibre} below, the  fibre over the point $c_i$  is a multiple fibre of 
type $2 I_{n_i}$, and setting $D ={\mu_2}_*\psi^* \Sigma$ we find a bisection of $f$, because $\Sigma $ is a section of $\alpha \colon W\to C$,   $\psi$ has degree $2$ and  conditions $(iv)$ (a) and (b) imply that $\Sigma$ is not contained in the branch locus $B$. 
\end{construction}

\begin{lem}\label{lem: double fibre}
 In Construction \ref{construction deg 2}, the fibre of $f$ over $c_i$ is a 
double fibre of type $2 I_{n_i}$. 
\end{lem}

\begin{proof}  Example \ref{ex: singularity for double fibre} shows 
that over each elliptic singularity, the exceptional divisor $E_i$ in the minimal resolution has self-intersection $-1$ and is either a smooth elliptic curve or a cycle of $n_i$ rational curves. Therefore, the set-theoretic fibre over $c_i$ on $\hat S$ is the union of $E_i$ and the reduced pullback $\hat\Gamma_i$ of the curve $\Gamma_i$. By Kodaira's classification of singular fibres (see Table \ref{tab:fibres}), this can only be the blow-up of a fibre of 
$I_{n_i}$ type and $\hat\Gamma_i$ has to be a $(-1)$ curve. Since $\Gamma_i$ is in the branch locus of $\bar\theta$, the 
scheme-theoretic fibre has multiplicity $2$, so that 
$f^{-1} (c_i)$ is a double fibre of type $2I_{n_i}$.
\end{proof}

As expected, we now prove that Construction \ref{construction deg 2} is roughly the inverse of taking the birational double-cover model.

\begin{prop}\label{prop: construction_equal_double_cover_model}
 In the situation of Construction \ref{construction deg 2}, the map $\bar \theta \colon \bar X \to W$ is the birational double cover model of the marked elliptic surface $\left( f \colon S \to C, \kl = \ko_S(F+ D)\right)$, where the surface $S$ is obtained as in (\ref{main diagram deg 2}), $F=\sum_{i=1}^k F_i$ and $W=\IP(f_{*}\kl)$.
\end{prop}

\begin{proof}
By Lemma \ref{lem: simple_base_pts}, each relative linear system $|\mathcal{L}|_{F_i}|$ has a unique, simple base point $q_i \in F_i$. Consider the ample model $X = \text{Proj}_C \mathcal{R}(f, \mathcal{L})$ and its blow-up $\sigma \colon \hat{X} \to X$ at the points $\varphi(q_i)$. Because each $\ko_{F_i}(F_i)$ is a non-trivial 2-torsion line bundle, the points $q_i$ do not lie in the bisection $D$. So $\sigma\colon  \hat X \to \bar X$ blows up $k$ points not lying on $\varphi(D)$. With notations as in Theorem \ref{thm: birational double cover model}, denote the strict transform of $\varphi(D)$ on $\hat X$ and $\bar X$ with $\hat D$ and $\bar D$, respectively. Then the only thing to verify is that under the induced birational morphism $\hat X \overset{\rho}{\to} \hat S$ we have an identification
 \[ 
 \hat \theta^* \ko_{W}(1) = \hat \theta ^* \Sigma = \bar \sigma^* \bar D  = \hat D + \hat G = \left( \hat D + \hat G + \hat E\right) - \hat E \simeq \mu_2^*\left(D + F\right) - \hat E,
 \]
or, equivalently, $\sigma^* \mathcal{L}(-\hat{E}) \cong \mathcal{O}_{\hat{S}}(\hat{D} + \hat{G})$. Indeed, this shows that the map $\hat \theta$ in Theorem \ref{thm: birational double cover model} coincides with $\psi\circ \rho$, where $\psi$ is as in Construction \ref{construction deg 2}. Here, $\hat G$ is as in the proof of Lemma \ref{lem: can bundle birat model}.
\end{proof}

Next, we express the invariants of the bisection in terms of the construction data.

\begin{lem}\label{lem: invariants bisection in birational model}
 In Construction \ref{construction deg 2} we have 
\[ D^2 = 2\deg(\det \ke_1) - k \text{ and } p_a(D) = 2 g(C)  + 2 \deg ( \det \ke_1) + \deg \kn -k-1. \]
\end{lem}
\begin{proof}
 Note that since $\mu_2$ blows up points away from $D$ and the exceptional curves satisfy $\hat{\Gamma}_i^2 = -1$, we can relate the pullbacks on the smooth surface $\hat{S}$ by:
\[ 
D^2 = (\mu_2^*D)^2 = \left( \mu_1^* \bar D - \tsum_{i=1}^k \hat \Gamma_i\right)^2 = (\mu_1^* \bar D)^2 + \tsum_{i=1}^k \hat \Gamma_i^2 = \bar D^2 - k.
\]
By the theory of projective bundles \cite[Appendix A]{LazarsfeldI} and since $\deg \bar \theta = 2$, we obtain 
\[ 
\bar D^2 = (\bar \theta^* \Sigma)^2 = 2 (\Sigma^2) = 2 \deg(\det\ke_1),
\]
which gives the claimed formula for the self-intersection.

To compute the arithmetic genus, we substitute the above formula, the result of Lemma \ref{lem: can bundle birat model} and the canonical 
bundle formula \eqref{eq: can bundle formula} into the adjunction formula on 
the smooth surface $S$ to get 
\begin{align*}
 2p_a(D) - 2 & = \left( K_S +D\right) .D\\
 &= f^*\left(K_C+ f_*K_{S/C}\right).D +  \left(\tsum_{i=1}^k F_i\right). 
D + D^2 \\
&= 2\left( 2g(C) -2 \right) 
+ 2 \deg  \left(\det \ke_1\tensor \kn \left(-\tsum_i c_i\right)\right) 
+ k +2\deg(\det\ke_1) -k\\
&=2(2g(C) -2) + 4 \deg ( \det \ke_1) -2k + 2 \deg\kn \\
& = 2\left (2 g(C)  + 2 \deg ( \det \ke_1) + \deg \kn -k-2 \right) 
 \end{align*}
 resulting in the above formula.
\end{proof}

\begin{rem}\label{rem: numbers for deg 2}
In applications where we want to construct specific elliptic fibrations, we are typically given the invariants of $S$, the number $k$ of multiple fibres, and the self-intersection and genus of the bisection $D$. The formulas from Lemmas \ref{lem: can bundle birat model} and \ref{lem: invariants bisection in birational model} allow us to determine these completely. We spell this out in the case where $q(S) = 0$, which implies that $f_*\omega_{S/C}$ is non-trivial and that the base curve is $C = \mathbb{P}^1$.

Setting
\[ 
a = \frac{D^2 +k}{2} \quad \text{and} \quad b = p_g(S) + 1 + k - a,
\]
we find that $\det \ke_1 = \ko_{\mathbb{P}^1}\left(a\right)$ and $\kn = \ko_{\mathbb{P}^1}\left(b\right)$.

Now, while the vector bundle $\ke_1$ decomposes as a direct sum of line bundles, the short exact sequence in Construction \ref{construction deg 2} $(ii)$ splits if and only if the extension class in $H^1(\mathbb{P}^1, \ko_{\mathbb{P}^1}(-a))$ vanishes, for example whenever $a \leq 1$. When the sequence splits, we can write $\ke_1 \cong \ko_{\mathbb{P}^1} \oplus \ko_{\mathbb{P}^1}(-a)$, and $W = \mathbb{P}(\ke_1)$ is the Hirzebruch surface $\mathbb{F}_a$, where the section $\Sigma$ corresponds to the unique curve with self-intersection $-a$. In particular, the surface $S$ can be realised birationally as a hypersurface of bidegree $(2b, 4)$ in the toric threefold considered in Example \ref{example: toric threefold} with $d_1=0$, $d_2=a$ and $d_3=b$. That is, birationally, $S$ is described by a global equation 
\begin{equation}\label{eq:global_toric_hypersurface_S}
y^2 = F_{4}(t_0, t_1, x_0, x_1),
\end{equation}
where $F_4$ is a polynomial of degree $4$ in the fibre variables $x_0, x_1$ and with coefficients in $\mathbb{C}[t_0, t_1]$. Note that the canonical bundle on the smooth surface $S$ then is given by $K_S = f^*\ko_{\mathbb{P}^1}(a+b-k- 2) + \sum_{i=1}^k F_i$.
\end{rem}

\section{The local description of the ample model}\label{sect: local models}
In preparation for the global determinantal model, we study the ample model $X$ of a marked elliptic surface $(f\colon S \to C, \kl)$ near a fixed fibre, which amounts to describing the localisation of the algebra $\kr(f,\kl)$ at a point of $C$. Since this question is local, for the remainder of this section we assume that $C$ is the spectrum of a DVR with closed point $c$ and generic point $\eta$, and that $f$ is non-isotrivial. Alternatively, in the analytic category, we restrict the fibration to a sufficiently small Euclidean neighbourhood of $c$. We denote the reduction of the central fibre $S_c\coloneqq f^{-1}(c)$ by $F$. There are two cases to consider. If $S_c$ is non-multiple, the description of $X$ is well-known, and we summarise it in Section \ref{sec: non-multiple fibres}. If $S_c$ is a double fibre, explicitly describing the ample model is more involved. This case was first considered by the second-named author in \cite{ThesisAnna} and is the content of Proposition \ref{prop: alg model double fibre} in Section \ref{sec:double-fibres}. Note that on a marked elliptic surface of degree two, every multiple fibre has multiplicity two.

\subsection{Non-multiple fibre}\label{sec: non-multiple fibres}

\begin{prop}\label{prop:types_sing_fibres}
 Assume that $S_c$ is not a multiple fibre. Then, as already stated, the ample model $X$ is a double cover of $W=\IP(\ke_1)$. Moreover, the branch curve $B$ encodes the Kodaira type of $S_c$ as in \cite[Table 2]{AB}.
 \end{prop}
 
\begin{proof}
 The surface $S$ is the relatively minimal model of a minimal resolution of $X$. The singularities of $X$ arise from contracting the fibre components of $S$ where the restriction of $\kl$ has degree zero. In our local setup, because $S_{\eta}$ is assumed smooth, any such singularities must lie over the special point $c$. By reversing Construction \ref{construction: equivariant resolution}, it is routine to establish a dictionary relating the Kodaira type of $S_c$ with the singularities of $B$ and the intersection of $B$ with the fibre $\Gamma$ of $W\to C$ over $c$. This is precisely the content of \cite[Table 2]{AB}. For the reader's convenience, we reproduce this data in Tables \ref{tab:types_fibres_1} and \ref{tab:types_fibres_2}.
\end{proof}

\begin{rem}
    Note that in this case we can equivalently interpret $X$ as a relative hypersurface of degree four in $\IP(1,1,2)\times C\to C$.
\end{rem}

\begin{table}[h!]
		\caption{Types of singular fibres when $\Gamma \not < B$}
		\label{tab:types_fibres_1}
		\begin{tabular}{lcc}
			\toprule
            $\Gamma \cap B$ & sing. of $B$ & Kodaira type of $S_c$ \\
            \midrule
            $b_1+b_2+b_3+b_4$ & none & $I_0$ \\
            $b_1+b_2+2b_3$ & $A_{k-1}, k\geq 1$ & $I_{k}$ \\
            $2b_1+2b_2$ & $A_{l-1}+A_{m-1}$, $l,m, \geq 1$ & $I_{l+m
            }$\\
            $b_1+3b_2$ & $A_{k}, k=0,1,2$ & $II\,(k=0)$, $III\,(k=1)$ or $IV\,(k=2)$ \\
             & $D_{k+4}$ ($k\geq 0$) & $I_{k}^*$ \\
             & $E_k$ ($k=6,7,8$) & $IV^*\,(k=6)$, $III^*\,(k=7)$ or $II^*\,(k=8)$ \\
            $4b_1$ & $A_k, k=0,1$ & $III\, (k=0)$ or $IV\,(k=1)$ \\
            & $A_{k+3}$ $(k\geq 0)$ & $I_{k}^*$ \\
            & $D_{k+3}, k\geq 1, k\neq 2$ & $I_{k}^*$ \\
            & $D_5$ & $I_2^*$ ($\Gamma\cap B^{sing}=\emptyset$) or $IV^*$\\
            & $E_6$ & $III^*$ \\
            \bottomrule
            \end{tabular}
            \end{table}

            \begin{table}
		\caption{Types of singular fibres when  $\Gamma  < B$}
		\label{tab:types_fibres_2}
		\begin{tabular}{lcc}
			\toprule
            $\Gamma \cap (\overline{B\backslash \Gamma})$ & sing. of $\overline{B\backslash \Gamma}$ & Kodaira type of $S_c$ \\
            \midrule
            $b_1+b_2+b_3+b_4$ & none & $I_0^*$ \\
            $b_1+b_2+2b_3$ & $A_{k-1}, k\geq 1$ & $I_{k}^*$ \\
            $2b_1+2b_2$ & $A_{l-1}+A_{m-1}$, $l,m, \geq 1$ & $I_{l+m
            }^*$\\
            $b_1+3b_2$ & $A_k, k=0,1,2$ & $IV^*\,(k=0)$, $III^*\,(k=1)$ or $II^*\,(k=2)$ \\
            $4b_1$ & $A_k, k=0,1$ & $III^*\,(k=0)$ or $II^*\,(k=1)$ \\
            \bottomrule
            \end{tabular}
            \end{table}

\subsection{Double fibre}\label{sec:double-fibres}

We now assume that $S_c=2F$ is a double fibre. We first describe the section ring associated with the restriction $\kl_c = \kl|_{S_c}$:
\begin{lem}\label{lem: section ring double fibre}
	We have an isomorphism $R(S_c, \kl_c) \isom  \IC[x_0, x_1, y_0, y_1, z]/I$, where the variables $x_0, x_1, y_0, y_1$ and $z$ have degrees $1,1,2,2$ and $3$, respectively, and $I$ is the ideal generated by the relations
	\[
    \rk \mat{ 0  & x_1 & y_1\\ x_1 & y_0 & z \\ y_1 & z & f_4 } \leq 1,
    \]
	for some polynomial $f_4(x_0, x_1, y_0)$ of weighted degree four that is not divisible by $y_0$ or $x_0^2$ modulo $x_1$.	Moreover,
    \begin{enumerate}
        \item the reduced curve $F$ is smooth if and only if $f_4$ is not a square modulo $x_1, y_1$, and
        \item the involution on $X$ acts as
	$ (x_0, x_1, y_0, y_1, z) \mapsto (x_0, x_1, y_0, -y_1, -z)$.
    \end{enumerate}
\end{lem}

\begin{proof}
We first observe that we may assume $F$ is irreducible. Indeed, if $F$ is a fibre of type $I_k$ with $k\geq 2$, we can pass to the ample model by contracting all but one of the rational components in the cycle. The corresponding section ring on $X_c$ then pulls back under $\varphi$ to the desired section ring $R \coloneq R(S_c, \kl_c) = R(2F, \kl|_{2F})$. 

Since the divisor $F$ is stable under the involution $\tau$, the functoriality of the restriction of equivariant sheaves to $\tau$-stable subschemes ensures that $\tau$ induces an involution on the graded ring $R$ and a compatible one on $R(F, \mathcal{L}|_F)$. We will use these induced actions as an organising principle. From the structure of the general fibre, we infer that the Hilbert series of the invariant subring $R^+$ is 
\[
(1-t)^{-2} = 1 + 2t + 3t^2 + \cdots,
\]
while the Hilbert series of $R$ is 
\[
(1-t^4)(1-t)^{-2}(1-t^2)^{-1} = 1 + 2t + 4t^2 + \cdots.
\]
In particular, up

Next, recall from Lemma \ref{lem:torsion} that we have a decomposition  $\ko_{2F} \cong \ko_F \oplus \ko_F(-F)$. In particular, twisting by $\kl^{\tensor n}$ yields a short exact sequence for each $n > 0$:
\[
0 \to H^0(F, \kl^{\tensor n}|_F - F) \to H^0(2F, \kl^{\tensor n}|_{2F}) \to H^0(F, \kl^{\tensor n}|_F) \to 0,
\]
which is equivariant with respect to the involution $\tau$.

		\begin{table}
		\caption{Generators in low degree for $R=R(2F, \mathcal{L}|_{2F})$}
		\label{tab: low deg}
		\begin{tabular}{lcc}
			\toprule
			line bundle & section & divisor\\
			\midrule 
			$\kl|_F$& $\bar x_0$ & $q_1$\\
			$\kl^{\tensor 2}|_F$& $\bar y_0$ & $2q_2$\\
			$\kl^{\tensor 3}|_F$& $\bar z$ & $q_2+q_3+q_4$\\
			\midrule
			$\kl|_F-F$& $x_1$ & $q_2$\\
			$\kl^{\tensor 2}|_F-F$& $\bar x_0x_1$ & $q_1+q_2$\\
			& $y_1 = {x_1 \bar z}/{\bar y_0}$ & $
			q_3+q_4$\\
			$\kl^{\tensor 3}|_F-F$& $\bar x_0^2x_1$ & $2q_1+q_2$\\
			& $\bar x_0y_1$ & $q_1+q_3+q_4$\\
			& $x_1\bar y_0$ & $ 3q_2$\\
			\bottomrule
		\end{tabular}
	\end{table} 
We now use this to carefully pick generators for $R$, collecting important information in  Table \ref{tab: low deg}.

We first choose generators for $R(F, \kl|_F)$. We start by picking $\bar x_0, \bar y_0 \in \IC(F)$ such that $H^0(F,\kl|_F)=\langle \bar x_0\rangle$ and $H^0(F,\kl^{\tensor 2}|_F) = \langle \bar x_0^2,\bar  y_0\rangle$. These two independent sections of $H^0(F,\kl^{\tensor 2}|_F)$ define a double cover $ F\to  \IP^1$ ramified over four points, say $q_1, q_2, q_3$, and $q_4$. If $F$ is nodal, two of these ramification points coincide, and we assume without loss of generality that $q_3 = q_4$ and that $\bar x_0$ vanishes at $q_1$. Furthermore, we choose $\bar y_0$ such that its divisor is $2q_2$. Now, the deck involution of the cover acts on the space $H^0(F,\kl^{\tensor 3}|_F)$. We choose a third generating section $\bar z$ for this space as a generator of the corresponding anti-invariant subspace. With these choices, the double cover structure is defined by the relation
\begin{equation}\label{eq: double cover relation F}
(\bar x_0\bar z)^2 - \bar x^2_0 \bar y_0 \tilde f_4(\bar x_0,\bar  y_0) = 0
\end{equation}
where $\tilde f_4$ is a weighted homogeneous polynomial of degree $4$. In particular, the divisor of $\bar z$ is $q_2+q_3 + q_4$, and the divisor of $\tilde f_4(\bar x_0,\bar y_0)$ is $2q_3+2q_4$.
We note that the involution induced by $\kl|_F$ on $F$ acts by 
\[ (\bar x_0, \bar y_0, \bar z)\mapsto  (\bar x_0, \bar y_0, -\bar z).\]

Next, we turn our attention to the sections in $R$ that vanish on $F$. Since $|\kl^{\tensor 2}|_F|$ is a base-point-free pencil on $F$ and $H^1(F, \kl^{\tensor (n-2)}|_F - F) = 0$ for all $n\geq 2$, the base-point-free pencil trick implies that the multiplication maps
\[
H^0( F, \kl^{\tensor 2}|_F) \otimes H^0(F, \kl^{\tensor n}|_F -F) \to H^0(F, \kl^{\tensor n+2}|_F - F)
\]
are surjective for all $n\geq 2$. Therefore, it suffices to analyse the graded pieces of degree at most three and the relations among the resulting generators. Up to relabeling, we may assume that $\ko_F(-F) \cong \ko_F(q_2-q_1)$. We then choose $x_1$ to be a generator of $H^0(F, \kl|_F-F)$, which consequently vanishes at $q_2$.

In $H^0(F,\kl^{\tensor 2}|_F-F)$, we obtain the linearly independent sections $\bar x_0x_1$ and $y_1 \coloneqq x_1\bar z/\bar y_0$. Note that $y_1$ is indeed a regular section because its divisor, $q_3+q_4$, is effective. Similarly, the three sections in $H^0(F, \kl^{\tensor 3}_F-F)$ listed in Table \ref{tab: low deg} are linearly independent, as can be verified by evaluating their divisors.

Finally, we choose lifts $x_0, y_0\in R^+$ and $z\in R^-$ of $\bar x_0, \bar y_0, \bar z \in R(F, \kl|_F)$, respectively. This yields the desired generators for the section ring. The complete set of generators is summarised in Table \ref{tab: low deg}.

	To show that these sections satisfy the relations defining the ideal $I$, first note that $x_1^2 = y_1^2 = x_1y_1 = 0$ in $R$, as all three sections vanish twice along $F$. Furthermore, the relation $x_1z = y_1y_0$ holds by the definition of $y_1$ and is independent of the choice of lifts precisely because of the three previous quadratic relations. In particular, the five sections satisfy a relation of the form
 \[
    z^2 = g_6(x_0, x_1, y_0, y_1) = y_0\tilde f_4(x_0,y_0) + x_0x_1 \tilde g_4(x_0,y_0) + y_1\tilde h_4(x_0,y_0),
    \]
    which reduces to \eqref{eq: double cover relation F} when restricted to $F$.
    However, $x_0, x_1, y_0 $ and $z^2$ are invariant while $y_1 = zx_0/y_0$ is anti-invariant, so $\tilde h_4 = 0 $.

    By adjusting the choice of the lift $y_0$ of $\bar y_0$ appropriately, we can ensure that the polynomial $g_6$ is divisible by $y_0$, which yields $z^2 = y_0f_4(x_0, x_1, y_0)$. More explicitly, letting $a=\tilde{f}_4(1,0)$ and $b=\tilde{g}_4(1,0)$, the new lift is given by $y_0 + \frac{b}{a}x_0x_1$. Note that $\tilde{f}_4$ is a homogeneous polynomial of degree two in $x_0^2$ and $y_0$; thus, $a$ is simply the coefficient of $x_0^4 = (x_0^2)^2$ and $b$ is defined analogously. We have thus shown that all elements of $I$ vanish in the ring $R$, and a straightforward verification confirms that no further relations exist.
    
    Given the choices made during the proof, we have also seen that the involution acts as stated. 
\end{proof}

   With the description of the section ring $R(S_c, \kl_c)$ in hand, we turn to the ample model of the marked elliptic surface $(f\colon  S\to C, \kl)$ in our local setting.

\begin{prop}\label{prop: alg model double fibre}
	For a suitable local coordinate $T$ of $C$ at the special point $c$, the ample model $X$ is embedded in $C \times \mathbb{P}(1^2, 2^2, 3)$ with coordinates $X_0, X_1, Y_0, Y_1, Z$ and ideal sheaf
	\[ 
    \ki = \left( \rk \mat{ T  & X_1 & Y_1\\ X_1 & Y_0 & Z \\ Y_1 & Z & F_4 } \leq 1\right),
    \]
	where $F_4$ is a section of $\ko(4)$. In particular, the quotient by the involution $X/\tilde \tau$ is given by $\{ TY_0 -X_1^2 = 0 \} \subset C \times \IP(1,1,2)$.
\end{prop}

\begin{proof}
	For each $n\geq 0$ and $k=1,2$, consider the restriction map 
	\[ 
    \rho_{n,k}\colon H^0 (C, \ke_n) \to H^0(kF, \kl^{\tensor n}|_{kF})
    \]
    and observe that in our local setting, the bundles $\ke_n$ are trivial.
	Choose generators for the section ring on $S_c$ as prescribed by 
Lemma~\ref{lem: section ring double fibre}, and fix a trivialization 
$\mathcal{E}_1 = X_0 \ko_C \oplus X_1 \ko_C$ satisfying $\rho_{1,2}(X_i) = x_i$. The sections $X_0^2, X_0X_1, X_1 ^2$ are linearly independent on every smooth fibre, but $X_1^2$ vanishes along $S_c = 2F$, implying it is 
divisible by $T$, where $T$ is an arbitrary but fixed uniformiser on $C$ at $c$ 
(i.e., a local function vanishing to order one at $c$ and nowhere else). It follows 
that the section $Y_0 = X_1^2/T$ does not vanish identically along the reduced fibre 
$F$. In fact, because $Y_0$ is a square on the generic fibre, its restriction to 
$F$ vanishes to order two at a single point. Consequently, we may choose 
$\bar{y}_0 = \rho_{2,1}(Y_0)$ as in the proof of Lemma~\ref{lem: section ring double fibre}. Note that $X_0^2, X_0X_1, Y_0$ are generators for $\ke_2^+$. Let now $Y_1$ be a generator for $\ke_2^-$, which necessarily restricts to zero on $F$ since $H^0(F, \kl^{\tensor 2}|_F)$ is invariant. Then, up to scaling, Lemma~\ref{lem: section ring double fibre} implies that $\rho_{2,2}(Y_1) = y_1$. With all these choices fixed, applying the same coordinate 
transformations as in the proof of Lemma~\ref{lem: section ring double fibre} 
further ensures that $\rho_{2,2}(Y_0) = y_0$.

  Now, the bundle $\mathcal{E}_3^+$ has rank four and is generated by 
\[
X_0^3, \quad X_0^2X_1, \quad X_0Y_0 = \frac{X_0 X_1^2}{T}, \quad X_1Y_0 = \frac{X_1^3}{T}.
\]
Indeed, their restrictions to each fibre are linearly independent. In $\mathcal{E}_3^-$, we have the section $X_0Y_1$, which is non-zero on $2F$ but vanishes along $F$, and the section $X_1Y_1$, which vanishes identically along $2F$. Dividing by $T$ yields an additional generator $Z = X_1Y_1/T$ for this bundle. Note that the restriction $Z|_F$ is a non-zero anti-invariant section; hence, after scaling by an appropriate constant, we have $\rho_{3,1}(Z) = z$.

It remains to analyse the relations. The section $Y_1^2$ belongs to $\mathcal{E}_4^+$ and can therefore be expressed as an $\ko_C$-linear combination of the generators 
\[
X_0^4, X_0^3X_1, X_0^2Y_0, X_0X_1Y_0, Y_0^2.
\]
That is, $Y_1^2 = G_4(X_0, X_1, Y_0)$ for some relative weighted homogeneous polynomial $G_4$ of degree $4$. Since $Y_1$ vanishes on $F$, its square $Y_1^2$ vanishes along the multiple fibre $2F$, which allows us to factor out the uniformiser and write $G_4 = T \cdot F_4(X_0, X_1, Y_0)$. Finally, we verify that the relations defining the ideal sheaf $\mathcal{I}$ hold by restricting to the fibres. On the generic fibre, we can divide by $T$ and eliminate the variables $Z$ and $Y_0$, leaving $Y_1^2 = TF_4$ as the single defining relation. On the special fibre, the generators and relations restrict precisely to those computed in Lemma~\ref{lem: section ring double fibre}. Furthermore, by our construction, it is clear that the invariant subring corresponds to the asserted quotient.
\end{proof}

\begin{rem}
 Let us explain how to read off the Kodaira type of $F$ directly from the equations, assuming that we have a fibre of type $I_k$ with $k>0$.   In the context of Lemma~\ref{lem: section ring double fibre}, we can choose $f_4 \equiv (x_0^2 - y_0)^2 \pmod{x_1}$, so that the branch points with respect to the pencil $(x_0^2 : y_0) \colon F \to \mathbb{P}^1$ are $0$, $\infty$, and a double branch point at $1$. Consequently, in the ample model presented in Proposition~\ref{prop: alg model double fibre}, the node of the central fibre $X_c$ is the point $q = (c,(1:0:1:0:0))$, which is contained in the affine open set where $X_0Y_0 \neq 0$. For simplicity, we set $X_0 = 1$ and treat $Y_0$ as a local unit. The relations $T = X_1^2/Y_0$ and $Y_1 = X_1Z/Y_0$ can then be used to eliminate two of the variables, yielding a single relation of the form:
\begin{align*} 
0  &= Z^2- Y_0 F_4(T, X_0, X_1, Y_0, Y_1)\\
& = Z^2- Y_0 \left( \left( X_0^2-Y_0\right)^2 + X_1g_3( X_0, X_1, Y_0) +T\left( \dots \right) \right) \\
& = Z^2- Y_0 \left( \left( 1-Y_0\right)^2 + X_1g_3( 1, X_1, Y_0) +X_1^2/Y_0\left( \dots \right) \right) \\
& = Z^2- Y_0  \left( 1-Y_0\right)^2 + X_1^k\left (\text{unit near $q$} \right)
\end{align*}
for some integer $k \geq 1$. This defines an $A_{k-1}$ singularity, which, when resolved, gives an $I_k$ fibre in the relative minimal model $S$.
\end{rem}

\begin{cor}
 	The fibrewise involution $\tilde \tau$ has exactly two isolated fixed points, both of which are smooth points of $X$. These are the unique base point $q_1$ of the linear system $|\kl|_F|$ and the unique point $q_2\in F$ such that $\ko_F(F) \simeq \ko_F(q_2-q_1)$. 
\end{cor}

\begin{proof}
It suffices to observe that the quotient $X/\tilde \tau$ has exactly two isolated $A_1$ singularities. Indeed, from the description $\{ TY_0 - X_1^2 = 0 \} \subset C \times \mathbb{P}(1,1,2)$, we see that these singularities are located at the points $(c, (1:0:0))$ and $(c, (0:0:1))$. These correspond precisely to the images of the points $q_1$ and $q_2$ in $F$, as detailed in the proof of Lemma~\ref{lem: section ring double fibre}.
\end{proof}

\begin{rem}
In our local setting, let us assume for simplicity that $F$ is smooth (of type $I_0$) so that $S$ and $X$ are equal. We observe that $\kl(F)$ is another relative polarisation of degree two, which agrees with $\kl$ on the generic fibre and thus induces the same global involution $\tau$ on $S$. We thus have a diagram
\[\begin{tikzcd}
 & \arrow[dashed]{dl}[swap]{\theta}\dar{2:1}\arrow[dashed]{dr}{\theta'} S \\
 \IP\left(f_*\kl\right) & S /  \tau \rar[dashed] \lar[dashed] & \IP\left( f_*\kl(F)\right)\\
 & \widetilde{S/\tau} \arrow{ul}{\sigma} \uar[swap]{\rho} \arrow{ur}[swap]{\sigma'}
\end{tikzcd}
\]
where the upper maps are (birationally) double covers, the map $\rho$ is the minimal resolution and $\sigma, \sigma'$ are compositions of two blow-ups. 
We illustrate this in Figure \ref{fig: quotient and both sides} below, where the blue curve, the strict transform of $F/\tau$ in $\widetilde{ S/\tau}$, contracts to a point in both ruled surfaces, so these differ by what is sometimes called an elementary transformation, i.e., a blow-up and a blow-down of one fibre. 

It is straightforward to generalise this observation to the case where the double fibre is of type $I_k$ with $k\geq 1$.
\end{rem}
\begin{figure}[t]
	\caption{Birational maps between  $S/\tau$ and the targets of the two birational double cover models with respect to $\kl$ and $\kl(F)$.}\label{fig: quotient and both sides}
	\begin{tikzpicture}[scale = .85,every node/.style = { font = \scriptsize} ]
		
		\begin{scope}[xshift = 0cm, yshift = 4.5cm] 
			\draw[rounded corners] (-2.0, 1.5) rectangle (1.5, -1.25);
			\node[] at (.8, 1.1) {\small$ S/\tau  $};

			\draw[very thick, blue] (-1, -1.25) to ++ (0, 2.75);
			\draw (-1, -1.25) to node[left] {$2F/\tau$} ++ (0, 2.75);

			\fill[red] (-1, -.5) circle (2pt);
			\fill[green!60!black] (-1, .75) circle (2pt);

			\node (A) at (0.5,0.125) {$A_1$ sing.};
			 \draw[->] (-.3, .225) to  ++(-.4, .4);
			 \draw[->] (-.3, .025) to  ++(-.4, -.4);

			\draw[pfeil, dashed] (2, 0) ->   ++ (1.5, 0); 
			\draw[pfeil, dashed] (-2.5, 0) ->   ++ (-1.5, 0);

		\end{scope}

		\begin{scope}[xshift = 6cm, yshift = 4.5cm]
			\draw[rounded corners] (-2.0, 1.5) rectangle (1.5, -1.25);
			\node[] at (.6, 1.1) {\small$ \IP\left(f_*\kl(F)\right) $};
			\draw[red, thick, dashed](-1, -1.25) to  node[left] { $\Gamma'$}++ (0, 2.75);

		\end{scope}

		\begin{scope}[xshift = -6cm, yshift = 4.5cm]
			\draw[rounded corners] (-2.0, 1.5) rectangle (1.5, -1.25);
\node[] at (.6, 1.1) {\small$ \IP\left(f_*\kl\right) $};

\draw[green!60!black, thick, dotted](-1, -1.25) to  node[left] { $\Gamma$}++ (0, 2.75);
		\end{scope}

		\begin{scope}[xshift = 0cm, yshift = 0cm]
			\draw[rounded corners] (-2.0, 1.5) rectangle (1.5, -1.25);
			\node[] at (.8, 1.1) {\small$ \widetilde{S/\tau}$};

			\draw[very thick, blue] (-1, -1.25) to ++ (0, 2.75);
			\draw (-1, -1.25) to node[left] {$-1$} ++ (0, 2.75);

			\draw[red, thick, dashed] (-1.2, -.5) to node[below] {$-2$} ++ (1.5, 0);
			\draw[green!60!black, thick, dotted] (-1.2, .75)  to node[above ] {$-2$} ++ (1.5, 0);
			
			\draw[pfeil] (2, 0.25) -> node[below right] {two blow-ups}  ++ (3.5, 2.5); 
			\draw[pfeil] (-2.5, 0.25) ->  node[below left] {two blow-ups}  ++ (-3.5, 2.5); 
			\draw[pfeil] (-.25, 1.75) -> node[right] {resolution}  ++ (0, 1.25); 
		\end{scope}

	\end{tikzpicture}
\end{figure}
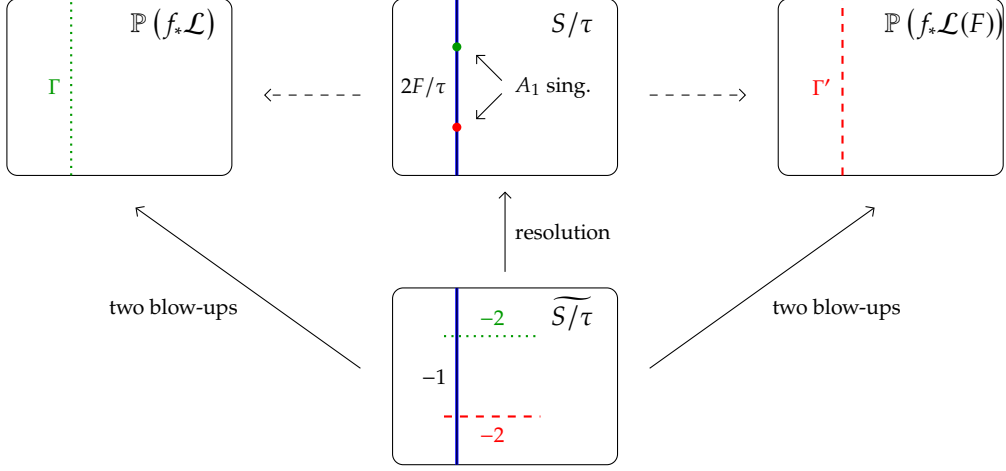

\section{Relative determinantal models inside weighted projective bundles}\label{sect:determinantal}
In this section we show how to glue the local description from Section \ref{sect: local models} to give a global embedding into a weighted projective bundle (Theorem \ref{thm:embed_double}). In favourable cases, which require a splitting condition, we can also write down the defining relations globally (Theorem \ref{thm: global determinantal models}).
\begin{defin}\label{def: splitting}
	A marked elliptic surface $(f\colon S\to C, \kl)$ of relative degree two satisfies the \newnotion{splitting condition} if $\ke_1=f_*\kl$ splits as a sum of line bundles, and a generating section of one of the summands vanishes along the reductions $F_1,\ldots, F_k$ of the multiple fibres (whenever they exist). That is, writing $\ke_1= \ka_0 \oplus \ka_1$, we require that $\ka_1 \subset f_* \kl( - \tsum_i F_i)$.
\end{defin}

In practice, the second condition in Definition \ref{def: splitting} above guarantees that the line bundle $\ka_1$ can globally assume the role that the coordinate $X_1$ plays in the local model of Proposition \ref{prop: alg model double fibre}, which is indeed necessary to globalise the local determinantal relations.

As a warm-up, let us first consider the classical case. For the remainder of this section, we fix a marked elliptic surface $(f\colon S\to C, \kl)$ of relative degree two. We adopt the convention that subsheaves of $\ke_1, \ke_2,$ and $\ke_3$ are denoted by $\ka_{\bullet}, \kb_{\bullet},$ and $\kc_{\bullet}$, respectively, and so on.

\begin{thm}\label{thm:embed_no_double}
	When $f$ has no double fibres, the ample model $X$ embeds as a relative hypersurface of degree four into the weighted projective bundle $ \IP_{(1,2)}\left(\ke_1, \ke_2^{-}\right)$. If, in addition, $f$ satisfies the splitting condition,  $\ke_1 = \ka_0 \oplus \ka_1$,  and we denote the global coordinates with $X_0, X_1, Y$, then we can write 
	  \[    X \cong \{Y^2 - F_4 = 0\} \subset \IP = \IP_{(1,1, 2)}\left(\ka_0, \ka_1, \ke_2^{-}\right)\]
	with $F_4 = \sum_{i = 0}^4 s_{i} X_0^i X_1^{4-i}$ and $s_i \in \Hom\left((\ke_2^-)^{\tensor 2}, \ka_0^i\tensor \ka_1^{4-i}\right)$. 
\end{thm}

\begin{proof}
According to the local description in Section \ref{sec: non-multiple fibres}, the weighted symmetric algebra $\text{wSym}_{(1,2)}\left(\ke_1, \ke_2^{-}\right)$ surjects onto the relative section ring $\kr(f, \kl)$, and its kernel is locally generated by a relative polynomial of degree four.

Since there are no double fibres, the splitting condition reduces to  $\ke_1 = 
\ka_0 \oplus \ka_1$. If we interpret the global coordinates as $X_i \colon \pi^* \ka_i \to \ko_{ \IP}(1)$ (see Section \ref{sect:wpb}), then to get a well defined relation $Y^2 - F_4\colon \pi^* \left( \ke_2^-\right) ^{\tensor 2} \to \ko_{\IP}(4)$, the $s_i$ are as claimed.
\end{proof}	

We now turn to the case where $f$ has $k \geq 1$ multiple fibres $2F_1, \ldots, 2F_k$. Here, it is convenient to treat the general case and the splitting case separately. 

\begin{thm}\label{thm:embed_double}
  If $f$ has at least one double fibre, then there exist line bundles $\kb_0$ and $\kc_0$ on $C$ such that the ample model $X$ embeds into the weighted projective bundle 
  \[
  \IP_{(1,2,3)}\left(\ke_1, \kb_0\oplus\ke_2^{-}, \kc_0\right)
  \]  as a relative determinantal variety, locally described by equations as in Proposition \ref{prop: alg model double fibre}. 
\end{thm}

\begin{proof}
We outline the argument from \cite[Theorem 3.12]{ThesisAnna} and split all the bundles $\ke_i = \ke_i ^+\oplus \ke_i ^-$ into invariant and anti-invariant parts. 

Recall that $\ke_1 = \ke_1^+$. Next consider the  short exact sequence,, see e.g., \cite[Proposition 3.11]{cp22},
\[
0 \to \Sym^{2}(\ke_1) \to \ke_2 \to \kt_2 \to 0,
\]
where $\kt_2$ is a torsion sheaf supported on the points $c_1, \ldots, c_k$. Since the base curve  $C$ is assumed projective, we can find a sufficiently negative line bundle $\kb_0$ with a surjection $\kb_0 \twoheadrightarrow \kt_2$ and such that $\Ext^1(\kb_0,\Sym^{(1)}(\ke_1)_2)=H^1(C,\kb_0^\vee\otimes \Sym^{(1)}(\ke_1)_2)$ vanishes. 
In other words, the chosen surjection can be lifted to a map $\kb_0 \to \ke_2^+$ and, since the invariant subring of every fibre is generated in degree at most two, we get a surjection
\[\wSym_{(1,2)}(\ke_1,\kb_0) \onto  \kr(f,\kl)^+.\]

For the anti-invariant part, we have that $\ke_1^- = 0$, and the line bundles $ \ke_2^-$ and $\ke_3^-$ fit into a short exact sequence
\[
0 \to \ke_1\tensor\ke_2^{-} \to \ke_3^- \to \kt_3\to 0.
\]
As before, we can choose $\kc_0$ such that $\kc_0 \twoheadrightarrow \kt_3$ and $H^1(C,\kc_0^\vee \otimes \kf_3)=0$, enabling us to lift the surjection to a map to $\ke_3^-$. This yields a surjection  
\[
\wSym_{(1,2,3)}(\ke_1, \kb_0\oplus\ke_2^{-}, \kc_0) \twoheadrightarrow \kr(f,\kl),
\]
because fibre-wise the section ring is generated in degree three. 
Taking $\Proj_C$ embeds the ample model into a weighted projective bundle with, locally, the same type of relations as given in Proposition \ref{prop:types_sing_fibres} and Proposition \ref{prop: alg model double fibre}.
\end{proof}

To treat the splitting case next, for a divisor $\Delta$ on $C$, we will denote by $\sigma_\Delta$ a section of $\ko_C(\Delta)$, unique up to scaling, whose divisor is $\Delta$. 

\begin{thm}\label{thm: global determinantal models}
Assume $f$ satisfies the splitting condition and has at least one double fibre.  Fix a general point $c\in C$ and let 
\begin{align*}
\kb_0& = \ko_C(-2mc) \tensor \left( \ke_2^+ /\left( \ka_0 \tensor \ke_1\right)\right) \\
 \kc_0&  = \ko_C(-mc) \tensor \left( \ke_3^- /\left( \ka_0 \tensor \ke_2^-\right)\right).
\end{align*}
Then for $m\geq 0$ sufficiently large, the relative section ring $\kr(f, \kl)$  is isomorphic to 
\[ \wSym^*_{(1,2,3)}\left(\ka_0 \oplus \ka_1; \kb_0 \oplus {\ke_2^-} ; \kc_0\right)/\left(  \rk \mat{ \sigma_{\sum_i c_i}  & \sigma_{mc}X_1 & Y_1\\ \sigma_{mc}X_1 & Y_0 & Z \\ Y_1 & Z & F_4 } \leq 1\right) \]for a section $F_4$ in the bundle 
\[\shom \left(\ka_0^{\tensor 4}\oplus \left (\ka_0^{\tensor 3}\tensor \ka_1\right)
\oplus \left (\ka_0^{\tensor 2}\tensor \kb_0\right)
\oplus \left (\ka_0\tensor \ka_1\tensor \kb_0\right) \oplus \kb_0^{\tensor 2}, {\ke_2^-}^{\tensor 2}\right).
\]

We can choose $m=0$ if the cohomology groups 
$H^1\left(\left( \ke_2^+ /\left( \ka_0 \tensor \ke_1\right)\right)^\vee\tensor  \left( \ka_0 \tensor \ke_1\right)\right)$ and $H^1\left( \left( \ke_3 ^-/ \ka_0 \tensor {\ke_2^-}\right)^\vee\tensor  \ka_0 \tensor {\ke_2^-}\right)$ both vanish.
\end{thm}
\begin{proof} 
Let $\Delta = \sum_i c_i$ as a divisor on $C$. We will freely use the local model described in Proposition \ref{prop: alg model double fibre}.

 We denote the generator corresponding to  $\ka_i$ by $X_i$ (compare Section \ref{sect:wpb}) and note that by our assumption $X_1|_{F_i} = 0$ for all double fibres. By the local model, ${\ke_2^-}$ is a line bundle on $C$, and for the even part we obtain a diagram
 \begin{equation}\label{eq: deg 2 invariant}
  \begin{tikzcd}
  & 0\dar &0 \dar\\
   & \ka_0 \tensor \ke_1 \rar[equal] \dar& \ka_0 \tensor \ke_1\dar\\
   0 \rar& \Sym^2(\ka_0 \oplus \ka_1) \dar\rar & \ke_2^+ \dar{q_2}\rar & \bigoplus_{i=1}^k \IC_{c_i} \rar \dar[equal]& 0 \\
   0\rar &\ka_1^{\tensor 2}\rar\dar& \ke_2^+ /\left( \ka_0 \tensor \ke_1\right)\dar \rar & \bigoplus_{i=1}^k \IC_{c_i} \rar & 0 
   \\
   & 0 & 0 
  \end{tikzcd}.
 \end{equation}
so that in particular
\begin{equation} \label{eq: deg of quotient in deg 2} 
 \deg \ke_2^+ /\left( \ka_0 \tensor \ke_1\right) = 2 \deg \ka_1 + k
 \end{equation}

Again, by the local model, the image of $q_2$ is a line bundle which locally corresponds to the variable $Y_0$. To have this variable available globally, we need a splitting, which is controlled by 
\[ \Ext^1\left ( \ke_2^+ /\left( \ka_0 \tensor \ke_1\right), \left( \ka_0 \tensor \ke_1\right)\right) = H^1\left(\left( \ke_2^+ /\left( \ka_0 \tensor \ke_1\right)\right)^\vee\tensor  \left( \ka_0 \tensor \ke_1\right)\right).\]
While this splitting might not exist, we can choose $m$ large enough such that $H^1(\kb_0^\vee \tensor \ka_0 \tensor \ke_1) = 0 $, so the splitting exists for the subsheaf $\kb_0$.  Since $c$ is general, and not one of the points $c_i$, we get a surjection
$ \begin{tikzcd}\Sym^2(\ka_0 \oplus \ka_1) \oplus \kb_0 \rar[->>] & \left(\ke_2\right) ^+\end{tikzcd}$
and by the local model
a surjection 
\[\begin{tikzcd}\wSym^*_{(1,2)}(\ka_0 \oplus \ka_1; \kb_0)\rar[->>] & \bigoplus_{k\geq 0 }\ke_k ^+  =  \kr(f, \kl)^+  \end{tikzcd}\]
with only relation $\sigma_{2mc}X_1^2 = Y_0 \sigma_{\Delta}$.

Let us look at the anti-invariant sections in degree three. As above, the local model implies that there is a diagram
\begin{equation}\label{eq: deg 3 anti-invariant}
  \begin{tikzcd}
  & 0\dar &0 \dar\\
   & \ka_0 \tensor \ke_2^- \rar[equal] \dar& \ka_0 \tensor \ke_2^-\dar\\
   0 \rar& \ke_1\tensor \ke_2^- \dar\rar & \ke_3 ^- \dar\rar & \bigoplus_{i=1}^k \IC_{c_i} \rar \dar[equal]& 0 \\
   0\rar &\ka_1\tensor \ke_2^-\rar\dar& \ke_3 ^- / \ka_0 \tensor \ke_2^-  \dar \rar & \bigoplus _{i=1}^k \IC_{c_i} \rar & 0 
   \\
   & 0 & 0 
  \end{tikzcd},
 \end{equation}
 which shows that 
 \begin{equation} \label{eq: deg of quotient in deg 3} 
 \deg \ke_3^- /\left( \ka_0 \tensor \ke_2^-\right) =  \deg \ka_1+ \deg \ke_2^- + k.
 \end{equation}

The quotient locally corresponds to the variable $Z$, and again we seek a splitting of the middle column, which is controlled by 
\[ H^1\left( \left( \ke_3 ^-/ \ka_0 \tensor {\ke_2^-}\right)^\vee\tensor  \ka_0 \tensor {\ke_2^-}\right).\]
If this does not vanish, it certainly does after restriction to the subsheaf $\kc_0$ for sufficiently large $m$ as above. 

Comparing with the local model, we have constructed a surjection from the given weighted symmetric $\ko_C$-algebra onto the relative section ring and need to consider the relations.

Denoting by $Y_0$ the generator  of $\kb_0$, by $Y_1$ the generator  of $\ke_2^-$,  and by $Z$ the generator  of $\kc_0$, we have by construction the relation 
\[ Z\sigma_{\Delta} = X_1Y_1 \sigma_{mc}.\]
Also, by the local model, $Y_1$ vanishes along $F_i$ but not $2F_i$ for all $i$ and defines the double cover induced by the involution on the general fibre. Therefore, it satisfies a relation
\begin{equation}\label{eq: double cover relation}
 Y_1^2 = F_4(X_0, X_1, Y_0) \sigma_\Delta 
\end{equation}
for an invariant polynomial of relative degree four. Combining the three relations found so far,  we get
\begin{align*}
 Z^2 &  = X_1^2 \frac{ Y_1^2 \sigma_{mc}^2}{\sigma_\Delta^2} =\frac{Y_0\sigma_\Delta}{\sigma_{2mc}}\frac{ F_4 \sigma_\Delta \sigma_{mc}^2}{\sigma_\Delta^2} = Y_0F_4,\\
\sigma_{mc} X_1Z & =   \frac{ X_1Y_1 \sigma_{mc}}{ \sigma_{\Delta}} X_1=Y_1\frac{ X_1^2 \sigma_{mc}}{ \sigma_{\Delta}} = Y_1 Y_0.\\
\end{align*}
 These are exactly the five relations defined in the sheaf of ideals, which in each fibre restrict to the correct relations in the section ring. This concludes the proof. 
\end{proof}

\begin{rem}\label{rem:comparison}
For completeness, let us briefly show how one can compare the birational double cover model $\bar X$ from Theorem \ref{thm: birational double cover model} with the determinantal model from Theorem \ref{thm:embed_double}. With $\kn^\vee= \ke_2^-\left(-\sum_i c_i\right)$, according to Lemma \ref{lem: can bundle birat model}, we get a sequence of inclusions of graded $\ko_C$-algebras
	\[ 
    \wSym_{(1,2)}\left(\ke_1, \kn^\vee\right)\into \wSym_{(1,2)}\left(\ke_1, \ke^-_2\right) \into \wSym_{(1,2)}\left(\ke_1, \kb_0 \oplus \ke_2^-, \kc_0\right),
    \]
	which, using Remark \ref{rem: birational model in wp bundle}, gives us a commutative diagram
	\[
	\begin{tikzcd}
		X\dar[dashed]{\theta} \rar[hookrightarrow] &  \IP_{(1,2,3)}\left(\ke_1, \kb_0 \oplus \ke_2^-, \kc_0\right)\dar[dashed]\\
		\bar X \rar[hookrightarrow] & \IP_{(1,2)}\left(\ke_1, \kn^\vee\right)
	\end{tikzcd}.
	\]
If we are in the situation of Theorem \ref{thm: global determinantal models}, we can read of the equation defining $\bar X$ by introducing a new variable $Y_1'\coloneq \sigma_{\sum_i c_i} Y_1$:  after restricting \eqref{eq: double cover relation} to the general fibre by substituting $ \sigma_{\sum_i c_i}Y_0 = \sigma_{2mc}X_1^2 $ and then expanding, we get the relation
\[ {Y_1'}^2 =  \sigma_{\sum_i c_i}^3F_4\left( X_0, X_1, \frac{\sigma_{2mc}X_1^2}{ \sigma_{\sum_i c_i}}\right) = \sigma_{\sum_i c_i}X_1^4 + \sigma_{\sum_i c_i}^2\left(\dots\right)\]
and we can see that the branch divisor has the description given in Theorem \ref{thm: birational double cover model}. 
\end{rem}
\subsection{Splitting fibrations over $\IP^1$}\label{sect:split_base_p1}

We now specialise to the case where the base curve $C$ is $\mathbb{P}^1$. We write $\ka_0 \cong \ko_{\mathbb{P}^1}(-a_0)$, $\ka_1 \cong \ko_{\mathbb{P}^1}(-a_1)$, and $\ke_2^{-} \cong \ko_{\mathbb{P}^1}(-b_1)$. Set $b_0=2(a_1+m) - k$ and $c_0=a_1 + b_1 + m - k$, we consider the toric $5$-fold $\IT=\IT(a_0,a_1,b_1,m,k)$ with weight matrix 
\[
\begin{pmatrix}
 t_0 & t_1 & x_0 & x_1 & y_0 & y_1 & z \\
 1 & 1 & a_0 & a_1 & b_0 & b_1 & c_0 \\ 
 0 & 0 & 1 & 1 & 2 & 2 & 3
\end{pmatrix},
\]
and irrelevant ideal $I = (t_0, t_1) \cap (x_0, x_1, y_0 , y_1, z)$. We define a natural involution on $\IT$ that acts as $(t_0, t_1, x_0, x_1, y_0, y_1, z) \mapsto (t_0, t_1, x_0, x_1, y_0, -y_1, -z)$. Under this setup, and with notations as in Theorem \ref{thm: global determinantal models}, we can prove the following.

\begin{cor}\label{cor: global toric model}
Assume $f\colon S \to \mathbb{P}^1$ satisfies the splitting condition of Definition \ref{def: splitting} and has $k$ double fibres. Assume these double fibres lie over the points $c_i \in \mathbb{P}^1$ cut out by the homogeneous binary form $g(t_0, t_1) = \prod_{i=1}^k (t_0 - c_i t_1)$. 
Then, for any 
\[ m\geq \max\left\{ 0,\,  \frac{k-1 -2(a_1-a_0)}2,\,   k-1 +a_0-a_1\right\} \]

 (and possibly smaller $m$), the ample model $X$ is embedded into the toric variety $\IT$ and described by the equations generated by
\begin{equation}\label{eq: global determinantal matrix over P1}
\rk \begin{pmatrix}
 g & t_1^m x_1 & y_1 \\
 t_1^m x_1 & y_0 & z \\  
 y_1 & z & f_4 
\end{pmatrix} \leq 1,
\end{equation}
where $f_4$ is an invariant section of bidegree $(2b_1 - k, 4)$. 

Conversely, if the section $f_4$ is chosen sufficiently general, then the equations generated by \eqref{eq: global determinantal matrix over P1} define the ample model of a relatively minimal elliptic fibration $f \colon S \to \mathbb{P}^1$ with $k$ double fibres over the designated points. 

Assuming $q(S) = 0$, the numerical invariants of the ample model are given by
\[ 
\chi(\ko_X) = b_1 -a_0 -a_1\geq 1 \text{ and } p_g(X) =
b_1 -a_0 -a_1-1.
\]
In particular,
\[
\omega_X\cong \tilde f^*\ko_{\IP^1}(b_1 -a_0 -a_1-2)\otimes \ko_X\left(\tsum_{i=1}^k G_i\right).
\]

\end{cor}

\begin{proof}
	Theorem \ref{thm: global determinantal models} yields an embedding of the relative section ring into a split weighted projective bundle over $\mathbb{P}^1$, which corresponds precisely to the toric variety $\IT$. 
	Since we are on $\IP^1$, the cohomology vanishings that guarantee the splitting in \eqref{eq: deg 2 invariant} and \eqref{eq: deg 3 anti-invariant} are controlled numerically; that is, we need the degrees of the relevant line bundles to be at least $-1$. 
	Using \eqref{eq: deg of quotient in deg 3}, we see that the condition in degree three is $m\geq k-1 +a_0-a_1$. In degree two, we get the conditions 
	$2m \geq k-1 -2(a_1-a_0)$ and $2m \geq  k-1 +a_0-a_1$ from \eqref{eq: deg of quotient in deg 2}. And since, $m\geq 0$, we obtain the constraint
    \begin{align*}
    m&\geq \max\left\{ 0,\,  \frac{k-1 -2(a_1-a_0)}2,\,   k-1 +a_0-a_1,\frac{k-1+a_0-a_1}{2}\right\} \\
    &=  \max\left\{ 0,\,  \frac{k-1 -2(a_1-a_0)}2,\,   k-1 +a_0-a_1\right\}.
    \end{align*}
	To compute the invariants, we recall from Lemma \ref{lem: can bundle birat model} that 
	\[
    \tilde f_*\omega_{X/\IP^1} =  \det \ke_1 \tensor\left( \ke_2^-\right)^\vee = \ko_{\IP^1}(b_1 -a_0 -a_1)
    \]
    
	and the stated formulas follow from the ones in Proposition \ref{prop: numerical invariants}.
\end{proof}

\begin{rem}\label{rem: Splitting condition one double fibre over genus zero}
If $f\colon S\to \mathbb{P}^1$ has at most one multiple fibre (i.e., $k \leq 1$), then the splitting condition of Definition \ref{def: splitting} is always satisfied. Indeed, by Grothendieck's theorem, the rank-two vector bundle decomposes as $\ke_1 \cong \ko_{\mathbb{P}^1}(-a_0) \oplus \ko_{\mathbb{P}^1}(-a_1)$. Without loss of generality, we can assume $-a_0 \leq -a_1$, which implies $a_0 - a_1 \geq 0$.

If there is no multiple fibre ($k=0$), the condition is vacuously satisfied. If there is exactly one multiple fibre over a point $c_1 \in \mathbb{P}^1$, any existing decomposition can be modified by an element of $\text{Hom}(\ko_{\mathbb{P}^1}(-a_0),\, \ko_{\mathbb{P}^1}(-a_1)) \cong H^0(\ko_{\mathbb{P}^1}(a_0 - a_1))$. Since this vector space is non-trivial, we can always choose a global section whose vanishing locus passes precisely through $c_1$, thereby satisfying Definition \ref{def: splitting}.
\end{rem}

\begin{rem}
Note that if there are no multiple fibres, then $k = 0$, $m = 0$, and $g(t_0,t_1) \equiv 1$. This allows us to explicitly eliminate the variables $y_0$ and $z$ from the determinantal equations because $y_0 = x_1^2$ and $z = x_1y_1$. Consequently, $X$ embeds as a relative hypersurface into a toric $3$-fold defined by a weight matrix
\begin{equation}
\begin{pmatrix}
 t_0 & t_1 & x_0 & x_1 & y_1 \\
 1 & 1 & a_0 & a_1 & b_1 \\  
 0 & 0 & 1 & 1 & 2 
\end{pmatrix},
\end{equation}
and is cut out by a single global equation of the form $y_1^2 - f_4(t_0, t_1, x_0, x_1) = 0$, where $f_4$ is an invariant polynomial of bidegree $(2b_1, 4)$.
\end{rem}

\section{A miscellany of examples}\label{sect:examples}
We now present a miscellany of examples that illustrate the scope and limitations of the relative determinantal models, including an explicit description for Halphen surfaces of index two. Enriques surfaces are considered separately in the next section.  

Examples of  birational double cover models can already be found in the literature \cite[Section~4]{horikawa}, 
\cite[Section~3.1.1]{do-rollenske22}, \cite[Section~2]{MZ},  \cite[Example 3.1]{ST} .

\subsection{Construction starting from an arbitrary ruled surface}\label{construction over arbitrary ruled surface}

Let $C$ be a smooth projective curve and let $\ke$ be any vector bundle of rank two, giving rise to a ruled surface $\alpha \colon W = \IP(\ke) \to C$.
Consider distinct fibres $\Gamma_1, \ldots, \Gamma_k$ and in each fibre $\Gamma_i$ pick a point $q_i$. Pick one more point $q_0\in  W \setminus \{q_1, \dots q_k\}$. Then consider the blow-up $ \Phi \colon \tilde  W \to  W $ at all these points and let $E_i$ be the exceptional curve over $q_i$. 

Now, on $C$ we consider a line bundle $\kn$ such that on $W$ we have 
\[
 \kb  \coloneq \alpha^*\kn^{\tensor2}(4) = \alpha^*\left ( \omega_C \tensor \det \ke \tensor \kl'\right) \tensor \ko_{ W }(4) = \omega _{ W } \tensor \ko_{ W }(6) \tensor \alpha^*\kl'
\]
where $\kl'$ is a sufficiently positive line bundle on $C$ to be chosen later. 
Then on the blow-up $\tilde  W$, the line bundle 
\[ 
\Phi^*\left(\ko_{ W }(6) \tensor \alpha^*\kl'\right)   \left(-2E_0 - 6\tsum_{i=1}^r E_i\right)
\]
is positive on all vertical curves. For $\kl'$ sufficiently positive, it is also positive on all horizontal curves and has positive square, hence it is big and nef. Therefore,
\[ \Phi^*\kb  \left(-E_0 - 5\tsum_{i=1}^r E_i\right) = \omega_{\tilde W } \tensor \Phi^*\left(\ko_{ W }(6) \tensor \alpha\kl'\right)   \left(-2E_0 - 6\tsum_{i=1}^r E_i\right)\]
has no higher cohomology independently of the choice of the point $q_0$.

 Pushing forward to $ W $, we interpret this as follows:
denoting with $\ki$ the ideal sheaf of the points $q_1, \dots q_k$, the restriction maps
\[ H^0(W, \kb) \to H^0 (W, \kb / \ki^5 \kb) \text{ and }  H^0 (W,  \ki^5 \kb) \to H^0 (W,  \ki^5 \kb\tensor \IC_{q_0})\]
are both surjective. Thus, we can prescribe any $4$-jet at the points $q_1, \dots, q_k$ and find a section $s$ of $\kb$ realising this $4$-jet. Moreover, the linear subsystem $s+ |\ki^5 \kb|$ has no base point at $q_0$, hence it is base-point-free on $W\setminus \{q_1, \dots, q_k\}$.  In particular, by Bertini, we can find a curve $B'$ in $|\kb|$ with the prescribed singularities at $q_1, \dots, q_k$ and smooth otherwise. In particular, using the branch divisor $B'+\sum_{i=1}^r \Gamma_i$  as in Construction \ref{construction deg 2}, we get an elliptic fibration with double fibres exactly over the points $c_i = \alpha(q_i)$ and $\ke_1 = \ke$.

\subsection{Examples not verifying the splitting condition}
We noted in Remark \ref{rem: Splitting condition one double fibre over genus zero} that the splitting condition is always verified if there is at most one double fibre and the base curve is $\IP^1$.  The following examples show that this is the only instance where we can guarantee it. 

\begin{exam}
 The bundle $\ke_1$ does not need to be split in general if the base curve $C$ is not $\IP^1$. An explicit example can be constructed as follows. Let $C$ be any curve of genus $g>0$ and consider any non-split bundle $\ke$ of rank two on $C$, for example corresponding to a class $0\neq \eta \in H^1(C, \ko_C) = \Ext^1_{\ko_C}(\ko_C, \ko_C)$. 
 Then, by Section \ref{construction over arbitrary ruled surface}, there exists a marked elliptic surface $(f\colon S\to C, \kl)$ with $\ke_1= \ke$ and any desired number of double fibres. 
\end{exam}

\begin{exam}\label{ex:genus1}
 Let $C$ be an elliptic curve and let $\ke = \ko_C(p) \oplus \ko (q)$ with $p\neq q$. Note that while $\ke$ is split, the splitting is unique, because $\Hom( \ko_C(p), \ko_C(q)) = 0 $. Now, by Section \ref{construction over arbitrary ruled surface}, there exists a marked elliptic surface $f\colon (S, \kl) \to C$ with $f_* \kl = \ke$ and one double fibre. Moreover, going back into that construction, the position of the singularity of the branch locus resulting in the double fibre can be chosen arbitrarily, so that we can ensure that neither of the two sub line bundles satisfies the condition in Definition \ref{def: splitting}. 
\end{exam}

\begin{exam}
 Here is yet another example that still admits a quite general global description despite not satisfying our splitting condition (Definition \ref{def: splitting}). Consider the toric variety with weight matrix given by 
\[ 
\mat{ t_0 & t_1 & x_0 & x_1 & y_0 & y_1 & z\\ 1 & 1 & 0 & 0 & b_1 & b_2 & c\\ 0 
	& 0 & 1 & 1& 2& 2& 3}.
    \]
and irrelevant ideal $I = (t_0, t_1) \cap (x_0, x_1, y_0, y_1, z)$. Inside this toric variety, consider the surface cut out by equations
	\[	 \rk 
	 \mat{
	 	 t_0t_1  & t_1 x_0 + t_0 x_1 & y_1\\
	 	t_1 x_0 + t_0 x_1 & y_0 & z \\ 
	 	  y_1 & z & f_4 } 
	 \leq 1
	\]
	for a suitably chosen $f_4(t_0, t_1, x_0, x_1, y_0)$. Then over both subsets $t_0\neq0$ and $t_1\neq 0$ we have the standard determinantal description from Section \ref{sect:determinantal}, but no trivialisation satisfies the splitting condition. 
\end{exam}

\subsection{Halphen surfaces of index two}\label{sec:Halphen}

We will now describe the determinantal model of a Halphen surface of index two for two convenient choices of relative polarisation. A Halphen surface of index two is a rational elliptic surface $f \colon S \to \IP^1$ possessing a unique double fibre $2F_1$ ($k=1$). Since $S$ is rational, we have $\chi(\ko_S) = 1$, $p_g(S) = 0$, and $q(S) = 0$. Moreover, by the canonical bundle formula for elliptic surfaces \cite[Chapter V, Theorem 12.1]{BHPV} (see also (\ref{eq: can bundle formula})), the canonical divisor is linearly equivalent to $K_S \sim -F_1$. 

We fix $f \colon S \to \IP^1$ and a $(-1)$-curve $D \subset S$, which is thus a bisection of $F$ and satisfies $p_a(D) = 0$, and $D \cdot F_1 = 1$. We consider two natural choices of relative polarisations of relative degree two associated with $D$, namely $\ko_S(D)$ and $\ko_S(D-K_S)$.

\subsubsection{The choice of marking $\kl = \ko_S(D)$}\label{subsec:case_I}

We first analyse the case where the relative polarisation is chosen to be $\kl = \ko_S(D)$. We begin with the following preliminary observation.

\begin{lem}\label{lem:Halphen_pushforward_I}
	In this setup, for any integer $n \geq 1$, the direct images on the base curve $C=\mathbb{P}^1$ satisfy
	\[ 
    h^0(\mathbb{P}^1, f_* \mathcal{L}^{\otimes n}) = 1 \quad \text{and} \quad h^1(\mathbb{P}^1, f_*\mathcal{L}^{\otimes n}) = \frac{1}{2}(n^2 - n).
    \]
\end{lem}

\begin{proof}
	First, we observe that $H^i(\mathbb{P}^1, f_*\mathcal{L}^{\otimes n}) \cong H^i(S, \ko_S(nD))$. Next, Riemann-Roch combined with Serre duality yields:
	\[ 
    h^0(S, nD) - h^1(S, nD) + h^0(S, -nD + K_S) = \frac{1}{2}nD \cdot (nD - K_S) + 1. 
    \]
	Since $D^2=-1$ and $D \cdot K_S = -D \cdot F_1 = -1$, the right-hand side reduces to $1 - \frac{1}{2}(n^2 - n)$. Now, for $n \geq 1$, $h^0(S, nD) = 1$  and $h^0(S, -nD + K_S) = 0$, yielding the desired formulas.
\end{proof}

Next, we set the coordinates on the base $\mathbb{P}^1$ such that the unique double fibre is over the point $c_1 = (0:1)$, cut out by $g(t_0, t_1) = t_0$. 

From Lemma~\ref{lem:Halphen_pushforward_I}, we see that $\ke_1 = f_*\kl$ has $h^0(\IP^1, \ke_1) = 1$ and $h^1(\IP^1, \ke_1) = 0$. Therefore, $\ke_1 \cong \ko_{\IP^1} \oplus \ko_{\IP^1}(-1)$. The trivial factor corresponds to the section defining $D$, which therefore does not vanish along the double fibre. As seen in Remark \ref{rem: Splitting condition one double fibre over genus zero}, the splitting condition is thus satisfied with $a_0=0$ and $a_1=1$. 

 From Lemma~\ref{lem:Halphen_pushforward_I}, we also have $h^1(\mathbb{P}^1, \mathcal{E}_2) = 1$, which is contributed entirely by the anti-invariant part $\mathcal{E}_2^-$ since   $\text{Sym}^2 \mathcal{E}_1 = \ko_{\mathbb{P}^1} \oplus \ko_{\mathbb{P}^1}(-1) \oplus \ko_{\mathbb{P}^1}(-2)$ and we have an exact sequence
\[
0 \longrightarrow \text{Sym}^2 \ke_1 \longrightarrow \ke_2^+ \longrightarrow \mathbb{C}_0 \longrightarrow 0.
\]
Thus,  $\ke_2^- \cong \ko_{\mathbb{P}^1}(-2)$, so $b_1 = 2$. In fact, we can alternatively deduce $b_1=2$ from the formulas for the invariants of $X$ (hence of $S$) from Corollary \ref{cor: global toric model}, which also shows that there is a global determinantal model with $m = 0$.
In particular, the missing numerical parameters for the toric variety $\IT=\IT(a_0, a_1, b_1, m, k) = \IT(0, 1, 2, 1, 1)$ are $b_0 = 2(a_1 + m) - k = 2(1 + 0) - 1 = 1$ and $c_0 = a_1 + b_1 + m - k = 1 + 2 + 0 - 1 = 2$. That is, the ambient toric $5$-fold we are interested in is defined by the weight matrix
\begin{equation}\label{eq:Halphen_Cox_I}
\begin{pmatrix}
 t_0 & t_1 & x_0 & x_1 & y_0 & y_1 & z \\ 
 1 & 1 & 0 & 1 & 1 & 2 & 2 \\ 
 0 & 0 & 1 & 1 & 2 & 2 & 3
\end{pmatrix}
\end{equation}

We have proved the following.

\begin{prop}\label{prop: RES_I}
	The ample model embedding $\varphi\colon S \dashrightarrow X \hookrightarrow \IT$ of a Halphen surface of index two with polarisation $\mathcal{L} = \ko_S(D)$ is defined globally by
	\begin{equation}\label{eq:Halphen_matrix_I}
	\text{rk} \begin{pmatrix} 
	t_0 &  x_1 & y_1 \\ 
	 x_1 & y_0 & z \\ 
	y_1 & z & F_4 
	\end{pmatrix} \leq 1
	\end{equation}
	where $\IT$ has Cox ring determined by (\ref{eq:Halphen_Cox_I}) and $F_4(t_0, t_1, x_0, x_1, y_0)$ has bidegree $(2, 4)$.
\end{prop}

\subsubsection{The choice of marking $\kl = \ko_S(D - K_S)$}\label {subsec:case_II}

We now turn to the second choice of relative polarisation. We choose $\kl = \ko_S(D - K_S)= \ko_S(D + F_1)$ and, again, we begin with the following calculation, whose proof we omit since the argument is the same as the one in the proof of Lemma \ref{lem:Halphen_pushforward_I}.

\begin{lem}\label{lem:Halphen_pushforward_II}
	In this setup, for any integer $n \geq 1$, the direct images on the base curve $C = \mathbb{P}^1$ satisfy
	\[
	h^0(\mathbb{P}^1, f_* \mathcal{L}^{\otimes n}) =\frac{1}{2}(n^2+n)+1  \quad \text{and} \quad h^1(\mathbb{P}^1, f_*\mathcal{L}^{\otimes n}) = 0.
	\]
\end{lem}

From Lemma~\ref{lem:Halphen_pushforward_II}, $h^0(\mathbb{P}^1, \ke_1) = 2$ and $h^1(\mathbb{P}^1, \ke_1) = 0$. Consequently, $\ke_1 \cong \ko_{\IP^1}\oplus \ko_{\IP^1}$, meaning $a_0 = a_1 =0 $. Thus, $\ke_2^- \cong \ko_{\IP^1}(-1)$, i.e., $b_1 = 1$. 

By Corollary \ref{cor: global toric model}  we can construct the global determinantal model with $m=0$ and $b_0 = -1$ and  $c_0 = 0$. That is, the $5$-fold $\IT$ in this case has weight matrix
\begin{equation}
\begin{pmatrix}
 t_0 & t_1 & x_0 & x_1 & y_0 & y_1 & z \\ 
 1 & 1 & 0 & 0 & -1 & 1 & 0 \\ 
 0 & 0 & 1 & 1 & 2 & 2 & 3
\end{pmatrix}
\end{equation}
and irrelevant ideal $I = (t_0, t_1) \cap (x_0, x_1, y_0, y_1, z)$. We then have the following description.

\begin{prop}\label{prop: RES_II}
The ample model embedding $\varphi\colon S \dashrightarrow X \hookrightarrow \IT$ of a Halphen surface of index two with polarisation $\mathcal{L} = \ko_S(D-K_S)$ is defined globally by
\begin{equation}\label{eq:Halphen_matrix_II}
	\text{rk} \begin{pmatrix} 
	t_0 &  x_1 & y_1 \\ 
	 x_1 & y_0 & z \\ 
	y_1 & z & F_4 
	\end{pmatrix} \leq 1
	\end{equation}
where $F_4(t_0, t_1, x_0, x_1, y_0)$ has bidegree $(1, 4)$.
\end{prop}

\section{Enriques surfaces}\label{sect: Enriques}
Let $S$ be an Enriques surface. It is well known; see, e.g., \cite[VIII.15]{BHPV}, that $S$ admits elliptic fibrations $f\colon S \to \IP^1$ with exactly two double fibres $2F_0$ and $2F_1$. The reduced double fibres are called half-pencils of $f$. 
In addition, there is always a bisection $D$ of $f$. After Horikawa, we distinguish two cases: the Enriques surface is called 
\newnotion{special} if it admits an elliptic fibration $f$, with a bisection $D$ which is a $(-2)$-curve. 
If such a fibration does not exist, then the bisection $D$ of $f$ is the half-pencil of another elliptic fibration and $S$ is called \newnotion{non-special}.

Before we look at the models in the individual cases, let us provide a few cohomology computations for later reference. 
\begin{lem}
 Let $S$ be an Enriques surface and $D$ a half-pencil on $S$. Then for $n\geq 0$
 $ \chi(nD) = 1$ and 
\[h^2(nD) = 0 \text { and }  h^0(nD)  = h^1(n D) +1  = 
\begin{cases} 
\frac n2 + 1 &  n \text{ even}\\ \frac {n-1}2 + 1 & n \text{ odd}                                          \end{cases}
\]
\end{lem}
\begin{proof}
 Let $D'$ be the corresponding half pencil such that $D- D' = K_S$. Then 
 $h^2(nD) = h^2( K_S + D' +(n-1)D) = h^0( -D' -(n-1) D) = 0 $. 
  For the rest, we use that $2D$ is an elliptic pencil, Riemann-Roch on $S$, the restriction sequence, and Lemma \ref{lem:torsion}.
\end{proof}

\subsection{Non-special Enriques surfaces}
Let $f\colon S \to \IP^1$ be a non-special Enriques surface, that is, there is a relative polarisation of degree two $\kl = \ko_S(D)$ consisting of a half-pencil (for another elliptic fibration on $S$, which has a second half-pencil $D'$).

Then writing $\ke_1 = f_*D = \ko(-a_0) \oplus \ko(-a_1)$ we see that by the projection formula 
\[h^0(\ke_1) =  h^0 (D)  = 1, \;  h^1( \ke_1) =h^1(D)  = 0 \implies (-a_0 , -a_1) = ( 0, -1) \]

We claim that $f$ satisfies the splitting condition. Let $\Gamma_c$ be the fibre of $f$ over $c$. Since it is easier to argue with sections, we consider the bundle $f_*\kl(1)\isom \ko(1) \oplus \ko$ instead of $f_*\kl$, that is, the pushforward of  $D + \Gamma_c$. 
By \cite[Chapter VIII, Lemma 17.1]{BHPV} we have  $K_S \sim F_0 - F_1 \sim D' - D$ and hence 
\[ D + \Gamma_c \sim D + 2F_0 \sim D'+ F_1- F_0 + 2F_0 = D'+ F_0 + F_1. \]
If we represent the first map in n the exact sequence
\[ 0 \to f_*\left(\kl(D+\Gamma_c -F-F')\right) = f_* D' \to f_* (D + \Gamma_c)  = \ke_1(1) \to \IC_0\oplus \IC_1 \to 0\]
by a matrix
\[
\begin{tikzcd}[ampersand replacement=\& ]
f_* D' \isom \ko\oplus \ko(-1) \rar{\begin{pmatrix} a & b \\ c & d\end{pmatrix}} \& \ko(1) \oplus \ko  = \ke_1(1)
\end{tikzcd}
\]
we see that the sections of the summand $\ko(1)$ correspond to divisors $D +\Gamma_c$ for some $c$ and the section of $f_*D'$ is not of this form. Therefore the entry $c$ is non-zero.

We can choose the image of the first summand of $f_*D'$ twisted by $\ko(-1)$ as the line bundle $\ka_1$ in the splitting condition, with $\ka_0$ being the line bundle corresponding to the section defining $D$. In other words, $f$ satisfies the splitting condition. 

From Theorem \ref{cor: global toric model} and $p_g (S) = q(S) = 0 $ we deduce that $\ke_2^- \isom \ko(-2)$ and that we may choose $m=0$, so that we have proved the following.
\begin{prop}\label{prop: non-special Enriques}
The ample model $X$ of a non-special Enriques surface $S$ is embedded in the toric variety
\[ \IT\colon \quad \mat{ t_0 & t_1 & x_0 & x_1 & y_0 & y_1 & z_0\\
1 & 1 & 0 & 1 & 0 & 2 & 1\\ 
0 & 0 & 1 & 1& 2& 2& 3}\]
with irrelevant ideal $(t_0, t_1) \cap ( x_0, x_1, y_0, y_1, z_0)$ and cut out by  equations
\[
 \rk \begin{pmatrix}
 t_0t_1 &  x_1 & y_1 \\
 x_1 & y_0 & z_0 \\  
 y_1 & z_0 & f_4 
\end{pmatrix} \leq 1,
\]
where $f_4$ is an invariant section of bidegree $(2, 4)$.  
\end{prop}

\subsection{Special Enriques surfaces}

We want to consider special Enriques surfaces $f\colon S \to \IP^1$ with polarisation $\kl = D+F_0$, where $F_0$ is a half-pencil of $f$ and the $(-2)$-curve $D$ is a bisection, because this is the relevant polarisation if we want to write them as degenerations of non-special Enriques surfaces.

Since $h^0 (\kl) = 1$ and $h^1(\kl) = 0 $ we again have
$\ke_1 \isom \ko\oplus \ko(-1)$ where the trivial summand corresponds to a section $\sigma_{D+F_0}$ vanishing along $D+F_0$. As above, we deduce $\ke_2^- = \ko(-2)$ from the invariants of $S$ (and its ample model $X$).

\begin{lem}\label{lem: special Enriques non splitting}
	The elliptic fibration $f$ on the special Enriques surface with this polarisation does not satisfy the splitting condition.
\end{lem}
\begin{proof}
As in the non-special case, we twist by a fibre, that is, we look at the sequence
\[ 0 \to D + F_1 \to D + F_0 + 2F_1 \to \kl(2F_1) |_{F_0 + F_1} \to 0. \]
Pushing forward, we get
\[ 0 \to \ko \oplus \ko(1) \to \ke_1(1) \to \IC_0 \oplus \IC_\infty \to 0\]

The trivial summand of the first sheaf corresponds to the divisor $D+ F_1$, and thus maps trivially to the trivial summand of $f_*\kl(1)$. Therefore, there is no splitting with a summand contained in the image, and $f$ does not satisfy the splitting condition. 
\end{proof}
Even with the splitting condition failing, we can nevertheless use Theorem \ref{thm: global determinantal models} to embed the ample model $X$ into a toric fourfold. 
However, it is more convenient to introduce more variables to keep the equations globally in a nice format.
\begin{table}\caption{Bundles and generators for special Enriques surfaces}\label{tab: special Enriques}
\begin{tabular}{cll}
\toprule
degree & bundles & variables \\
\midrule
1 & $\ka_0 = \ko$, $\ka_1 = \ko(-1)$ & $x_0$, $x_1$\\
2 & $\kb_0 = \ko(1)$, $\kb_1 = \ko(-1)$, & $u_0$, $u_1$\\
& $\ke_2^- = \ko(-2)$ & $y_1$ \\ 
3 & $\kc_0 = \ko(-1)$, $\kc_1 = \ko(-2)$, & $z_0$, $z_1$\\
\bottomrule
\end{tabular}
\end{table}
\begin{lem}\label{lem: coordinates special Enriques}
We can choose splittings $\ke_1 = \ka_0 \oplus \ka_1$ and $\ke_2^+ = \kb_0 \oplus \ko(-1) \oplus \kb_1$ and $\ke_3^- = \kc_0\oplus \kc_1$ with bundles as in Table \ref{tab: special Enriques}
such that in sequence \eqref{eq: deg 2 invariant} the map is given as
	\[ 
	\begin{tikzcd}[ampersand replacement=\&, column sep=large]
		\Sym^2\ke_1 = \ko \oplus \ko(-1) \oplus \ko (-2)
		\rar { \begin{pmatrix} t_0 & 0 & 0 \\ 0 & 1 & 0\\ 0 & 0 & t_1\end{pmatrix}}\&
		 \kb_0\oplus \ko(-1) \oplus \kb_1 = \ke_2^+
	\end{tikzcd} \]
	and in sequence  \eqref{eq: deg 3 anti-invariant} the map is given as
		\[ 
	\begin{tikzcd}[ampersand replacement=\&, column sep=large]
		\ke_1\tensor \ke_2^- = \ko(-2)  \oplus \ko(-3)  
		\rar { \begin{pmatrix} t_0 & 0  \\ 0 & t_1\end{pmatrix}}\&
		\kc_0\oplus \kc_1 = \ke_3^-
	\end{tikzcd}. \]
\end{lem}
\begin{proof}
Consider the restriction sequence
\[ 0 \to D+F + F' \to D + F + 2F'  \to D|_{F'} \to 0. \] 
This is also exact on global sections, because, for example, the section defining $D+3F$ restricts to a non-zero section on $F'$ and $h^0(F', D_{F'}) = 1$. So there is a pencil of sections vanishing at least once along $F'$. Subtracting $F'$ again, we see that there is a unique section vanishing twice along $F'$. 

We have seen in the proof of Lemma \ref{lem: special Enriques non splitting} that the sections of the summand $\ka_0 (1) \subset \ke_1(1)$ correspond to divisors of the form $D + F + \Gamma_c$ of which only one vanishes along $F'$ at all. So there is a section  $\ka_1 \to \ke_1$ vanishing once along $F'$ and giving rise to the desired splitting in degree one. 

For the invariant part in degree two, we see that, similarly to \eqref{eq: deg 2 invariant}, we can mod out by $\ka_0 \tensor \ka_1$ to get a diagram 
 \begin{equation}\label{eq: deg 2 invariant special Enriques}
  \begin{tikzcd}
  & 0\dar &0 \dar& 0 \dar\\
  0 \rar & \ka_0 ^{\tensor 2} \rar \dar& \kb_0 \rar \dar& \IC_0 \dar\rar & 0 \\
   0 \rar& \ka_0^{\tensor 2} \oplus \ka_1^{\tensor 2} \dar\rar & \ke_2^+/(\ka_0\tensor \ka_1) \dar\rar &  \IC_{0 } \oplus \IC_\infty\rar \dar& 0 \\
   0\rar &\ka_1^{\tensor 2}\rar\dar& \kb_1\dar \rar & \IC_\infty \rar\dar & 0 
   \\
   & 0 & 0 & 0 
  \end{tikzcd}.
 \end{equation}
where we use that $x_0^2 $ vanishes along $\Gamma_0 = 2F_0$ and $x_1^2$ vanishes along $\Gamma_\infty = 2 F_1$. Looking at the degrees, it is clear that the middle column splits and that we can also split $\ke_2^+$. 

Starting from diagram  \eqref{eq: deg 3 anti-invariant}, the argument for the anti-invariant part in degree three is similar.
\end{proof}
Having chosen these coordinates, we can describe the embedding.
\begin{prop}\label{prop: special Enriques}
The  choice of coordinates in Table \ref{tab: special Enriques} embeds the ample model 	$X$ of a special Enriques surface
 into the toric variety
\[ \IT'\colon \quad \mat{ t_0 & t_1 & x_0 & x_1 & u_0 & u_1 & y_1 & z_0 & z_1\\
	1 & 1 & 0 & 1 & -1 & 1 & 2 & 1 & 2 \\ 
	0 & 0 & 1 & 1& 2& 2& 2& 3& 3}\]
with ideal generated by the twelve equations
\[
 \rk \begin{pmatrix}
	t_0 &  x_0 & y_1 \\
	x_0 & u_0 & z_0 \\  
	y_1 & z_0 & t_1f_4 
\end{pmatrix} \leq 1, \qquad 
z_0z_1 = x_0x_1f_4, \qquad 
 \rk \begin{pmatrix}
	t_1 &  x_1 & y_1 \\
	x_1 & u_1 & z_1 \\  
	y_1 & z_1 & t_0f_4 
\end{pmatrix} \leq 1.
\]
\end{prop}
\begin{proof} This is very similar to the proof of Theorem \ref{thm: global determinantal models}. 
By Lemma \ref{lem: coordinates special Enriques} we have the relations 
\[ x_1^2 = t_1 u_1 \text{ and }  x_0^2 = t_0 u_0\]
in the invariant subring and 
\[ x_0 y_1 = t_0 z_0 \text{ and }  x_1 y_1 = t_1 z_1\]
in the degree three anti-invariant part. 
By the local model, we have $y_1^2$ vanishing along the double fibres and the branch locus, that is,
\[ y_1^2 = t_0 t_1 f_4.\]

To verify the remaining relations, we compute
\begin{align*}
  z_i ^2& = \frac{x_i ^2 } {t_i^2} y_1^2 = \frac{x_i ^2 } {t_i^2} t_0 t_1 f_4 =  = \frac{t_i u_i } {t_i^2} t_0 t_1 f_4 =   \frac{t_0 t_1 } {t_i} u_i f_4, \\
 z_i x_i  &= \frac{ x_i y_1}{t_i} x_i = \frac{ t_i u_i y_1 }{t_i } = u_i y_1,\\
 z_i y_1 & = \frac{ x_i y_1}{t_i} y_1 = \frac{ x_i t_0 t_1f_4}{t_i},\\
 z_0z_1&  = \frac{ z_0 u_1 y_1}{x_1} = \frac{ x_0 t_0 t_1f_4}{t_0}\frac{ u_1}{x_1} = x_0f_4\frac{  t_1 u_1}{x_1} = x_0 x_1 f_4.
\end{align*}
This gives the claimed presentation, with six equations from each matrix, one extra equation and the equation $y_1^2 = t_0 t_1 f_4$ appearing twice in the list. 
\end{proof}

\subsection{Connecting special and non-special Enriques surfaces}
From Proposition \ref{prop: non-special Enriques} we see that non-special Enriques surfaces are parametrised by an equation $f_4$ contained in an open subset of a linear system. Below we will explain how, using our relative algebraic models, special Enriques surfaces can be deformed to non-special ones and thus reprove the following well-known result \cite[Thm. 4.3]{horikawa}.
\begin{thm}[Horikawa]
 Every special Enriques surface can be deformed to a non-special Enriques surface, and thus all Enriques surfaces are deformation equivalent. 
\end{thm}

Instead of starting with the final result, let us explain the process of finding the deformation step by step.

\subsubsection{Deforming the invariant subring}\label{sect: invariant deformation Enriques}
The invariant subrings can be deformed into each other using the so-called scrollar deformation trick: consider inside the toric variety 
\[ \tilde \IT^+\colon \quad \mat{ t_0 & t_1 & x_0 & x_1 & v_1& u_1 \\
	1 & 1 & 0 & 1 &  0 & 1 \\ 
	0 & 0 & 1 & 1&  2& 2}\]
with irrelevant ideal $(t_0, t_1)\cap (x_0, x_1,  v_1, u_1)$ the family cut out by 
\begin{align*}
	t_0 v_1&  =  t_1 v_0 - 2 \lambda x_0 x_1+\lambda^2u_1,\\
	x_0^2 &= v_0, \\
	x_1^2 &= t_1u_1,
\end{align*}
where $v_0$ is just a placeholder variable. Let us analyse the fibres.
\begin{prooflist}
	\item[The special fibre $\lambda = 0$]
	This is the key step in the scrollar deformation trick: the variable
	\begin{equation}\label{eq: u0} 
	u_0 := \frac{v_1}{t_1} = \frac{v_0} {t_0}	 
	\end{equation}
	is globally defined and can be used to eliminate $v_i = t_i u_0$
	The remaining equations are 
	\[ x_0^2 = t_0 u_0 , \, x_1^2 = t_1 u_1\]
	which are exactly the invariant equations for the special Enriques surfaces in Proposition~\ref{prop: special Enriques}. 
	\item[The general fibre $\lambda \neq 0$]
	We can eliminate $v_0$ again and also $u_1$, because $\lambda \neq 0$.
	After elementary algebraic manipulations, the remaining equation becomes

\begin{equation}\label{eq: enriques general invariant relation}
 \begin{split}
 t_0t_1v_1 & = t_1 ( t_0 v_1) \\
 & = t_1\left( \lambda^2 u_1 + t_1 x_0 ^2 - 2 \lambda x_0 x_1\right) \\
 & = \lambda^2 t_1 u_1 + t_1^2 x_0 ^ 2 - 2\lambda  x_1 t_1 x_0\\
 & = \lambda^2 x_1^2 - 2\lambda  x_1 t_1 x_0+ t_1^2 x_0 ^ 2\\
 & = \left(\lambda x_1 - t_1 x_0 \right)^2  
 \end{split}
\end{equation}
Up to a change in coordinates, this is the invariant subring for the non-special Enriques from Proposition \ref{prop: non-special Enriques}.

Note that the extra summand $2\lambda x_0 x_1$ in the first relation is not needed to make the ambient toric varieties deform to each other but rather to make the equations of the surfaces inside them come out nicely. 
\end{prooflist}

\subsubsection{Deforming the double cover}\label{sect: Stephens family}
Our own attempts to lift the scrollar deformation of the previous section to the double cover failed, and we are indebted to Stephen Coughlan for explaining to us how to obtain the family given below. Some of the claims that follow have been verified using  Macaulay2 \cite{M2}.\footnote{The relevant Macaulay2 code is included in the comments in the .tex file. }

Inspired by the scrollar deformations and since the anti-invariant variables $y_1$ and $z_0$ are shared between the special and non-special models (with the same weights and roles), as an ambient toric variety we are led to consider 
\[
\widetilde{\mathbb{T}}=
\begin{pmatrix}
t_0 & t_1 & x_0 & x_1 & v_1 & u_1 & y_1 & z_0 & z_1 \\
1 & 1 & 0 & 1 & 0 & 1 & 2 & 1 & 2 \\
0 & 0 & 1 & 1 & 2 & 2 & 2 & 3 & 3
\end{pmatrix}
\]
with irrelevant ideal $(t_0, t_1) \cap ( x_0, x_1, v_1, u_1, y_1, z_0, z_1)$.

We can naively try to lift  the twelve equations for the special Enriques surface from Proposition \ref{prop: special Enriques} to $\tilde 
\IT$, resulting in nine original equations not involving $u_0$

\begin{gather*}
 \rk \begin{pmatrix}
	t_1 &  x_1 & y_1 \\
	x_1 & u_1 & z_1 \\  
	y_1 & z_1 & t_0f_4 
\end{pmatrix} \leq 1
\end{gather*}
\vspace*{-1\baselineskip}
\begin{align*}
&x_0y_1 = t_0z_0 && y_1z_0 = t_1x_0f_4 && z_0z_1 = x_0x_1f_4 
\intertext{and for each of the three equations involving $u_0$ two new equations,}
&v_0  = x_0^2  && \boxed{ t_0x_0z_0 = v_0y_1 } && \boxed{ t_0z_0^2 = t_1v_0f_4 } \\
 &t_0 v_1  = t_1x_0^2  && t_1x_0z_0 = v_1y_1 && z_0^2 = v_1f_4, 
\end{align*}
one for each possible substitution from \eqref{eq: u0}. 

However, these equations do not define the special Enriques surfaces alone, because the primary decomposition of the ideal gives another component contained in the irrelevant ideal of $\tilde \IT$. 
Removing that superfluous component results in dropping the two boxed equations and adding $x_0x_1z_0 = v_1z_1$.

One can now carefully lift the deformation of the invariant subring found in Section \ref{sect: invariant deformation Enriques} to the double cover. A posteriori, they can be arranged nicely as follows after introducing some new variables:
\begin{equation}\label{eq: full deformation Enriques} 
\begin{matrix}
 v_0  = x_0^2, \qquad x_1^\lambda = t_0 x_0 -\lambda x_1, \qquad u_1^\lambda = x_0x_1 - \lambda u_1\\
 \\
 t_0 v_1  = t_1x_0^2 - 2\lambda x_0x_1 + \lambda^2 u_1 = \frac{ u_1^\lambda x_1^\lambda}{x_1},
 \qquad
x_0y_1= t_0z_0 + \lambda z_1, \\
 \\
\rk \begin{pmatrix}
	t_1 &  x_1 & y_1 \\
	x_1 & u_1 & z_1 \\  
	y_1 & z_1 & t_0f_4 
\end{pmatrix} \leq 1,
\quad
 \rk \begin{pmatrix}
	t_0x_1  &  x_1^\lambda & y_1 \\
	u_1^\lambda & v_1 & z_0 \\  
	z_1 & z_0 & f_4 
\end{pmatrix} \leq 1,
 \end{matrix}
\end{equation}

Flatness can be checked using \cite{M2} and \cite[Chapter III, Proposition 9.7]{Hartshorne}.
It therefore remains to analyse the fibres.
\begin{prooflist}
	\item[The special fibre $\lambda = 0$] This can be checked by hand, since most of the equations can immediately be spotted after setting $\lambda =0$ and substituting $v_i = t_i u_0$ from \eqref{eq: u0}. However, the resulting ideal does again differ from the one given in Proposition \ref{prop: special Enriques} by an irrelevant component, easily identified using Macaulay2.

\item[The general fibre $\lambda \neq 0$] 
We need to  eliminate the two extra variables
\begin{equation*}\label{eq: substitution general fibre}  
\lambda^2 u_1  =  t_0 v_1-  t_1x_0^2 +  2\lambda x_0x_1 \text{ and } 
\lambda z_1= x_0y_1-  t_0z_0 
\end{equation*}
and then show that with $ x_1^\lambda = t_1x_0 -\lambda x_1$ the relations are 
	\[
 \rk \begin{pmatrix}
 t_0t_1 &  x_1^\lambda & y_1 \\
 x_1^\lambda & v_1 & z_0 \\  
 y_1 & z_0 & f_4 
\end{pmatrix} \leq 1.
\]
This is easily done in Macaulay2.

We have found a flat deformation from the model of the special Enriques surfaces to the model of the non-special ones and thus concluded the proof of the theorem. \qed

\end{prooflist}

\bibliographystyle{alpha}
\bibliography{references}

\end{document}